\documentclass[english]{smfart}
\usepackage{irregular-vanishing}

\begin{document}
\frontmatter
\title{Kodaira-Saito vanishing for the~irregular~Hodge filtration}

\author[C.~Sabbah]{Claude Sabbah}
\address{CMLS, CNRS, École polytechnique, Institut Polytechnique de Paris, 91128 Palaiseau cedex, France}
\email{Claude.Sabbah@polytechnique.edu}
\urladdr{https://perso.pages.math.cnrs.fr/users/claude.sabbah}


\begin{abstract}
After making correct, and then improving, our definition of the category of irregular mixed Hodge modules thanks to Mochizuki's recent results, we show how these results allow us to obtain Kodaira-Saito-type vanishing theorems for the irregular Hodge filtration of irregular mixed Hodge modules.
\end{abstract}

\maketitle
{\let\\\xspace\tableofcontents}
\mainmatter

\section{Introduction}
\subsection{Corrigenda to \cite{Bibi15}}
In the memoir \cite{Bibi15}, we have defined the category $\IrrMHM(X)$ of irregular mixed Hodge modules on a complex manifold $X$, and we have proved various properties of this category. However, the definition of the category $\IrrMHM$ in \loccit\ has a flaw, as was noticed by T.\,Mochizuki, because the operation of rescaling that is used does not preserve coherence. In order to preserve coherence, one has to restrict the family of objects in order to ensure an algebraicity property in the twistor parameter $\hb$. We make the definition of \cite{Bibi15} correct in~Section \ref{subsec:IrrMHMprovisional}.

\subsection{Application of the results of \cite{Mochizuki21} to $\IrrMHM$}
On the other hand, the results of T.\,Mochizuki in \cite{Mochizuki21} allow us to improve the definition of $\IrrMHM$ and to avoid the above restriction of objects. We explain this improvement and the consequences of the results of \cite{Mochizuki21} to $\IrrMHM$ in Section \ref{subsec:resultsMochizuki}.

\subsection{An irregular Kodaira-Saito vanishing theorem}
Let $X$ be a smooth projective variety of dimension $n$. Let $\cT$ be an object of $\IrrMHM(X)$ (\eg with respect to the improved definition of Section \ref{subsec:improved}) and let $\ccM$ be the underlying holonomic (left) $\ccD_X$-module. Then $\ccM$ can be equipped in a canonical way with a coherent filtration indexed by $A+\ZZ$ for some finite set $A\subset[0,1)$, called the \emph{irregular Hodge filtration} and denoted by $F^\irr_\bbullet\ccM$. Given $\alpha\in A$, we~denote by $F^\irr_{\alpha+\ZZ}\ccM$ the corresponding $\ZZ$\nobreakdash-indexed filtration. The shifted holomorphic de~Rham complex $\pDR\ccM$ is thus filtered by setting
\[\let\to\ra
F^\irr_{\alpha+p}\pDR\ccM=\{0\to F^\irr_{\alpha+p}\ccM\to\Omega^1_X\otimes F^\irr_{\alpha+p+1}\ccM\to\cdots\to\underset{\cbbullet}{\Omega^n_X\otimes F^\irr_{\alpha+p+n}\ccM}\to0\},
\]
where $\cbbullet$ indicates the term in degree zero. We will show how the results of T.\,Mochizuki \cite{Mochizuki21}, combined with the original proof of M.\,Saito \cite{MSaito87},\footnote{One can also refer to \cite{Popa16,P-Sch11,P-Sch14,Schnell16,Schnell19} for proofs and applications. We will follow the proof given in \cite[\S11.9]{S-Sch}.} lead to the following vanishing result and its corollaries:

\begin{theoreme}[Kodaira-Saito vanishing]\label{th:vanishing}
Let $L$ be an ample line bundle on $X$ and let $\cT$, $\ccM$ and $A$ be as above. Then, for each $\alpha\in A$, we~have
\begin{align*}
H^i(X,\gr^{F^\irr_{\alpha}}\pDR(\ccM)\otimes L)&=0\quad \text{for }i>0,\\
H^i(X,\gr^{F^\irr_{\alpha}}\pDR(\ccM)\otimes L^{-1})&=0\quad \text{for }i<0.
\end{align*}
\end{theoreme}

\begin{corollaire}\label{cor:vanishing3}
For $\cT$, $\ccM$ and $A$ as above, let $a_o\in A+\ZZ$ be such that $F^\irr_{<a_o}\ccM=0$, and let us set $\omega_X=\Omega^n_X$. Then we have the vanishing
\[
H^k(X,\omega_X\otimes F^\irr_a(\ccM)\otimes L)=0\quad \forall k>0,\text{ and }\forall a\in[a_o,a_o+1).
\]
\end{corollaire}

Note that we can replace $F^\irr_a(\ccM)$ by $\gr^{F^\irr}_a(\ccM)$ in this corollary. We also have the analogue of Kollár's vanishing theorem:

\begin{corollaire}[Kollár vanishing for the irregular Hodge filtration]\label{cor:vanishing4}
Let $\cT$, $\ccM$, $A$ and $a_o$ be as in Corollary \ref{cor:vanishing3}. Let $f:X\to Y$ be a projective morphism to a smooth projective variety $Y$ and let $L$ be an ample line bundle on $Y$. Then we have the vanishing
\[
H^k(Y,R^jf_*(\omega_X\otimes F^\irr_a\ccM)\otimes L)=0\quad\forall j,\;\forall k>0,\text{ and }\forall a\in[a_o,a_o+1).
\]
\end{corollaire}

\pagebreak[2]
\begin{remarques}\label{rem:vanishing}\mbox{}
\begin{enumerate}
\item\label{rem:vanishing1}
The theorem will be proved for all objects of the category $\Cresc(X)$ introduced by T.\,Mochizuki in \cite{Mochizuki21} (\cf Section \ref{subsec:compmochi}), as the various compatibilities with functors needed for the proof are proved for this category and not for $\IrrMHM(X)$. One could also apply the result to objects of the category of exponential mixed Hodge modules introduced in \loccit
\item\label{rem:vanishing2}
For each $a\in A+\ZZ$, let us denote by ${<}a$ the predecessor of $a$ and let us set $\gr^{F^\irr}_a\ccM=F^\irr_a\ccM/F^\irr_{<a}\ccM$. We will show in Appendix \ref{app:B} that the statement of Theorem~\ref{th:vanishing} also holds for the $A+\ZZ$-graded complexes $\gr^{F^\irr}\pDR(\ccM):=\bigoplus_{a\in A+\ZZ}\gr^{F^\irr}_a\ccM$.
\item\label{rem:vanishing3}
One can also prove the vanishing of $H^k\bigl(X,R^jf_*(\omega_X\otimes\gr^{F^\irr}_a(\ccM))\otimes L\bigr)$ for~$k$,~$j$ and~$a$ as in Corollary \ref{cor:vanishing4} (\cf Appendix \ref{app:B}).
\end{enumerate}
\end{remarques}

\subsection{Geometric consequences}
Let us emphasize an example where Theorem \ref{th:vanishing} applies. Let $L$ be an ample line bundle on $X$ and let $D\subset X$ be a divisor with normal crossings. Recall that the Kodaira-Norimatsu vanishing theorem (\cite{Norimatsu78}): for each integer $p\geq0$,
\begin{equation}\label{eq:KSlog}
\begin{aligned}
H^q(X,\Omega_X^p(\log D)\otimes L)&=0\quad \text{for }p+q>n,\\
H^q(X,\Omega_X^p(\log D)\otimes L^{-1})&=0\quad \text{for }p+q<n.
\end{aligned}
\end{equation}

Theorem \ref{th:vanishing} enables us to extend this vanishing result when $D$ contains the support of the pole divisor~$P$ of a morphism $\varphi:X\to\PP^1$. The sheaf $\Omega^p(\log D,\varphi)$ is defined as the subsheaf of $\Omega_X^p(\log D)$ consisting of logarithmic forms $\omega$ such that $\rd \varphi\wedge\omega$ remains logarithmic. More generally, for each $\alpha\in[0,1)\cap\QQ$ and each $p\geq0$, we~define $\Omega^p(\log D,\varphi,\alpha)$ as the subsheaf of $\ccO_X(\floor{\alpha P})\otimes_{\ccO_X}\Omega_X^p(\log D)$ consisting of germs of meromorphic $p$-forms~$\omega$ such that $\rd \varphi\wedge\omega$ is a local section of $\ccO_X(\floor{\alpha P})\otimes_{\ccO_X}\Omega_X^{p+1}(\log D)$.

Let us decompose the reduced divisor $P_\red$ as $\bigcup_{i\in I}P_i$ and let $m_i$ be the multiplicity of $P_i$ in $P$. For each $\alpha\in[0,1)\cap\QQ$, let us set $I_\alpha=\{i\in I\mid \alpha m_i\in\NN\}$ and $P_\alpha=\bigcup_{i\in I_\alpha}P_i$. In what follows, it is enough to consider that $\alpha$ belongs to the finite subset $A\subset[0,1)\cap\QQ$ consisting of those $a$ such that $am_i\in\NN$ for some $i\in I$. If $\alpha>0$, we denote by ${<}\alpha$ the predecessor of $\alpha$ in $A$. For $\alpha\in A\cap(0,1)$ and $p\geq0$, the quotient sheaf $\Omega^p(\log D,\varphi,\alpha)/\Omega^p(\log D,\varphi,{<}\alpha)$ is supported on the divisor~$P_\alpha$.

\begin{corollaire}[of Theorem \ref{th:vanishing}]\label{cor:vanishing}
With the above assumptions and notations, for each $p\geq0$, the sheaves $\Omega^p(\log D,\varphi,\alpha)$ ($\alpha\in A$) satisfy the Kodaira-Saito vanishing property analogous to \eqref{eq:KSlog}.
\end{corollaire}

In particular, we obtain the vanishing $H^k(X,\omega_X(D+\lfloor\alpha P\rfloor)\otimes L)$ for $k>n$.

\begin{remarque}
By using Remark \ref{rem:vanishing}\eqref{rem:vanishing2}, one also obtains that the property of Corollary \ref{cor:vanishing} holds for the sheaves $\Omega^p(\log D,\varphi,\alpha)/\Omega^p(\log D,\varphi,{<}\alpha)$ ($\alpha\in A\cap(0,1)$).
\end{remarque}

\begin{corollaire}[of Corollary \ref{cor:vanishing4}]\label{cor:vanishing5}
Let $\varphi:X\to\PP^1$ be a projective morphism and set $P=\varphi^*(\infty)$. Assume that the support of $P$ is contained in a (reduced) divisor with normal crossings $D$ in $X$. Let $f:X\to Y$ be a projective morphism to a smooth projective variety~$Y$ and let $L$ be an ample line bundle on $Y$. Then for each $\alpha\in A$ we have the vanishing property
\[
H^k(Y, R^jf_*\omega_X(D+\lfloor\alpha P\rfloor)\otimes L)=0\quad\text{for all $k>0$ and all $j$}.
\]
\end{corollaire}

\subsection{Notations and conventions}\label{subsec:nota}
All over the paper, we will use the following notations and conventions.
\begin{enumerate}\setcounter{enumi}{-1}
\item\label{nota:0}
Unless otherwise stated, the sheaves of modules are left modules over their corresponding sheaf of ring.

\item\label{nota:1}
Let $X$ be a complex manifold and let $\ccD_X$ be the sheaf of holomorphic differential operators on $X$, with its filtration $F_\bbullet\ccD_X$ by the order. We denote by $R_X$ the Rees ring $R_F\ccD_X:=\bigoplus_kF_k\ccD_X\hb^k$, and by $\wt R_X$ the sheaf $R_X\langle\hb^2\partial_\hb\rangle$.\footnote{In \cite{Bibi15}, we use the notation $R_X^\mathrm{int}$ and $\cR_\cX^\mathrm{int}$; here, we adopt the notation $\wt R_X$ of \cite{Mochizuki21}.} We have the inclusions of sheaves of rings on $X$:
\[
R_X\subset\wt R_X\subset\ccD_X\otimes_\CC\CC[\hb]\langle\partial_\hb\rangle.
\]
Inverting the action of $\hb$ leads to the ring $R_X(*\hb):=\ccO_X[\hb,\hbm]\otimes_{\ccO_X[\hb]}R_X$ which is isomorphic to $\ccD_X\otimes_\CC\CC[\hb,\hbm]$. Similarly, $\wt R_X(*\hb)$ is isomorphic to the ring of differential operators $\ccD_X\otimes_\CC\CC[\hb,\hbm]\langle\hb^2\partial_\hb\rangle$. The inclusions above become
\[
R_X(*\hb)\subset\wt R_X(*\hb)=\ccD_X\otimes_\CC\CC[\hb,\hbm]\langle\partial_\hb\rangle.
\]

\item\label{nota:2}
If we consider the complex manifold $\cX:=X\times\CC_\hb$ with sheaf of holomorphic functions $\cO_\cX$ and projection $\pi:\cX\to X$, we have the analytic versions of the previous rings by applying the transformation $\cO_\cX\otimes_{\pi^{-1}\ccO_X[\hb]}\cbbullet$. We denote by $0_\hb$ the divisor $X\times\{0\}\subset\cX$. We obtain the sheaves
\[
\cR_\cX\subset\wt\cR_\cX\subset\cD_\cX\quand\cR_\cX(*0_\hb)\subset\wt\cR_\cX(*0_\hb)=\cD_\cX(*0_\hb).
\]
We denote by $\cX^\circ$ the open subset \hbox{$\cX\moins 0_\hb$}.

\item\label{nota:3}
We consider the partial projective completion $\fX=X\times\PP^1$ of $\cX$, so that $\cX$ is a Zariski open subset of $\fX$. We simply denote by $0_\hb$, \resp $\infty_\hb$, the hypersurfaces $X\times\{0\}$, \resp $X\times\{\infty\}$ of $\fX$. We have $\cX=\fX\moins\infty_\hb$ and $\cX^\circ=X\times\CC^*=\fX\moins(0_\hb\cup\infty_\hb)$. Similarly, $\fX^\circ$ denotes the open subset $\fX\moins 0_\hb$. Then $\fX$ is covered by the open sets~$\cX$ and $\fX^\circ$ with intersection $\cX^\circ$. Since the sheaf $\wt\cR_{\cX^\circ}$ is identified with $\cD_{\cX^\circ}$, we can extend it as $\cD_{\fX^\circ}(*\infty_\hb)$ on $\fX^\circ$. Then $\wt\cR_\cX$ can be extended as $\wt\cR_\fX\subset\cD_{\fX}(*\infty_\hb)$, so~that
\[
\wt\cR_\fX|_{\cX}=\wt\cR_\cX\subset\cD_\cX,\quad\wt\cR_\fX|_{\fX^\circ}=\cD_{\fX^\circ}(*\infty_\hb),\quad\wt\cR_\fX(*0_\hb)=\cD_{\fX}(*(0_\hb\cup\infty_\hb)).
\]
We also denote by $\pi$ the projection $\fX\to X$.

\item\label{nota:3b}
Assume that $X=Y\times\Delta_t$ for some disc $\Delta_t\subset\CC$. Each sheaf of rings considered above is equipped with an increasing filtration $V_\bbullet$ such that $t$ has order $-1$, $\partiall_t$ has order $1$ and operators defined on $Y$ have order zero. In such a way, the operators of $V$-order zero form a subsheaf of rings $V_0$.

\item\label{nota:3c}
The following proposition will be used various times.

\begin{proposition}\label{prop:subcoh}
Let $\cZ$ be any of the spaces $\cX,\fX$ and let $\cA_\cZ\subset\cB_\cZ$ by any pair of nested subsheaves of rings considered in \eqref{nota:2}--\eqref{nota:3b}. Let $\cN$ be a coherent $\cB_\cZ$-module which the union of a sequence of coherent $\cO_\cZ$-submodules $\cN_i$. Let $\cN'\subset\cN$ be a coherent $\cA_\cZ$-submodule of $\cN$. Then each $\cO_\cZ$-module $\cN'\cap\cN_i$ is coherent and $\cN'=\bigcup_i(\cN'\cap\cN_i)$ is also the union of a sequence of coherent $\cO_\cZ$-submodules.
\end{proposition}

\begin{proof}
It is enough to check that each $\cN'\cap\cN_i$ is coherent, and this is a local question on~$\cZ$, so we can assume that $\cN'$ is the union of a sequence of coherent $\cO_\cZ$\nobreakdash-mod\-ules~$\cN'_j$. We first claim that $\cN_i\cap\cN'_j$ is $\cO_\cZ$-coherent. It is enough to prove that $\cN_i+\cN'_j\subset\cN$ is $\cO_\cZ$-coherent since $\cN_i\cap \cN'_j$ is isomorphic to the kernel of $\cN_i\oplus \cN'_j\to \cN_i+\cN'_j$. Since $\cB_\cZ$ is $\cO_\cZ$-pseudo-coherent (\cf\cite[Def.\,A.5(1) \& Lem.\,A.26]{Kashiwara03}), it is enough to check that $\cN_i+\cN'_j$ is locally $\cO_\cZ$-finitely generated, which is clear.

From the claim we deduce that the increasing sequence $(\cN_i\cap \cN'_j)_j$ is locally stationary since $\cN_i$ is $\cO_\cZ$-coherent. Therefore, $\cN_i\cap \cN'=\bigcup_j (\cN_i\cap \cN'_j)$ is $\cO_\cZ$\nobreakdash-co\-herent.
\end{proof}

\item\label{nota:4}
For the rescaling operation, we denote by $\tau$ the rescaling variable. First, we~consider the algebraic description of the rescaling operation. We consider the sheaf of rings $\ccO_X[\tau,\taum,\hc]$ with the morphism of rings\vspace*{-3pt}\enlargethispage{\baselineskip}
\begin{align*}
\ccO_X[\hb]&\to\ccO_X[\tau,\taum,\hc]\\
\hb&\mto\hc\taum.
\end{align*}
By identifying $\ccO_X[\tau,\taum,\hc]$ with $\ccO_X[\hb][\tau,\taum]$ via the above correspondence, one sees that $\ccO_X[\tau,\taum,\hc]$ is $\ccO_X[\hb]$-flat. For the sake of clarity, we make a distinction between the variables $\hb$ and $\hc$, although both play the role of the twistor variable, before, respectively after, rescaling. We set\vspace*{-3pt}
\begin{equation}\label{eq:tauRX}
\tauR'_X=R_F\cD_X[\tau],\quad\tauR_X=R_F\cD_X[\tau]\langle\partiall_\tau\rangle,\quad\wt\tauR_X=\tauR_X\langle\hc^2\partial_\hc\rangle,
\end{equation}
which are sheaves on $X$ and where $\hc$ is the twistor variable. In local coordinates, $\tauR'_X$ and $\tauR_X$ read\vspace*{-3pt}
\[
\tauR'_X=\ccO_X[\tau,\hc]\langle\partiall_x\rangle,\quad \tauR_X=\ccO_X[\tau,\hc]\langle\partiall_x,\partiall_\tau\rangle,
\]
where, on the right-hand sides, $\partiall_x=\hc\partial_x$ and $\partiall_\tau:=\hc\partial_\tau$. We identify $\iota^*_{\tau=\hc}\tauR'_X:=\tauR'_X/(\tau-\hc)\tauR'_X$ with $R_F\cD_X$ that we can also write as $\sum_k\tau^k\otimes F_k\cD_X$.
We also set\vspace*{-3pt}
\begin{equation}\label{eq:tauRXtaum}
\begin{aligned}
\tauR'_X(*\tau)&=\ccO_X[\tau,\taum,\hc]\langle\partiall_x\rangle,\\
\tauR_X(*\tau)&=\ccO_X[\tau,\taum,\hc]\langle\partiall_x,\tau\partiall_\tau\rangle,\\
\wt\tauR_X(*\tau)&=\tauR_X(*\tau)\langle\hc^2\partial_\hc\rangle=\ccO_X[\tau,\taum,\hc]\langle\partiall_x,\tau\partiall_\tau,\hc^2\partial_\hc\rangle.
\end{aligned}
\end{equation}

\item\label{nota:5}
For the analytic version of the rescaling operation, we decorate the spaces with the letter $\tau$ on the left up side. So, we set\vspace*{-3pt}
\[
\begin{aligned}
\tauX&=X\times\CC_\tau,\\
\ov\tauX&=X\times\PP^1_\tau,
\end{aligned}
\qquad
\begin{aligned}
\taucX^\circ&=\cX^\circ\times\CC_\tau,\\
\ov\taucX^\circ&=\cX^\circ\times\PP^1_\tau,
\end{aligned}
\qquad
\begin{aligned}
\taucX&=\cX\times\CC_\tau,\\
\ov\taucX&=\cX\times\PP^1_\tau,
\end{aligned}
\qquad
\begin{aligned}
\taufX&=\fX\times\CC_\tau,\\
\ov\taufX&=\fX\times\PP^1_\tau.
\end{aligned}
\]
We also consider the divisors $0_\tau,\infty_\tau,0_\hc,\infty_\hc$ in $X\times\PP^1_\tau\times\PP^1_\hc=\ov\taufX$ with an obvious meaning, and we denote similarly their restrictions to the above open subsets $\taucX,\taucX^\circ,\taufX$ of $\ov\taufX$.

Any of the projections $\tauX\to X$, $\taucX\to\cX$ or $\taufX\to\fX$ (omitting $\tau$) is denoted by~$p$, while any of the projections $\cX\to X$, $\taucX\to\tauX$ or $\taufX\to\tauX$ (omitting $\hb$ or $\hc$) is denoted by $\pi$. Any of the projections $\taucX\to X$ and $\taufX\to X$ is denoted by $q$.

We consider the following sheaves of rings on $\ov\taufX$, whose sheaf-theoretic pushforward by $q_*$ gives back the sheaves of rings \eqref{eq:tauRX} and \eqref{eq:tauRXtaum} on $X$. By definition, any sheaf~$\cR$ or~$\wt\cR$ below satisfies $\cR=\cR(*(\infty_\tau\cup\infty_\hc))$. We thus consider
\[
\begin{aligned}
\cR_{\ov\taufX/\PP^1_\tau}&\quad\text{with $q_*\cR_{\ov\taufX/\PP^1_\tau}=\tauR'_X$},\\
\cR_{\ov\taufX}&\quad\text{with $q_*\cR_{\ov\taufX}=\tauR_X$},\\ \wt\cR_{\ov\taufX}&\quad\text{with $q_*\wt\cR_{\ov\taufX}=\wt\tauR_X$},
\end{aligned}
\quand
\begin{aligned}
\cR_{\ov\taufX/\PP^1_\tau}(*0_\tau)&\quad\text{with $q_*\cR_{\ov\taufX/\PP^1_\tau}=\tauR'_X(*\tau)$},\\
\cR_{\ov\taufX}(*0_\tau)&\quad\text{with $q_*\cR_{\ov\taufX}=\tauR_X(*\tau)$},\\
\wt\cR_{\ov\taufX}(*0_\tau)&\quad\text{with $q_*\wt\cR_{\ov\taufX}=\wt\tauR_X(*\tau)$},
\end{aligned}
\]
and we denote similarly their restrictions to the open subsets $\taucX,\taucX^\circ,\taufX$ of $\ov\taufX$. All these sheaves are subsheaves of rings of $\cD_{\ov\taufX}(*(0_\tau\cup0_\hc\cup\infty_\tau\cup\infty_\hc))$.

\item\label{nota:6}
In both the algebraic and the analytic settings, the above sheaves are equipped with their $V$-filtration with respect to the function~$\tau$, that we denote by $\tauV_\bbullet$. We have
\begin{align*}
\tauV_k(\tauR_X)&=
\begin{cases}
\tau^{-k}\,\tauR'_X\langle\tau\partiall_\tau\rangle&(k\leq0),\\
\sum_{j=0}^k\partiall_\tau^k\,\tauV_0\tauR_X&(k\geq1),
\end{cases}
\\
\tauV_k(\wt\tauR_X)&=\tauV_k(\tauR_X\langle\hc^2\partial_\hc\rangle)\quad \forall k\in\ZZ.
\end{align*}
For the $\tau$-localized modules, the filtration is simply the $\tau$-adic filtration made increasing. The definition is similar for the analytic sheaves $\cR$ and we have the relations of the form $q_*\tauV_k\cR=\tauV_k\tauR$.

\item\label{nota:7}
Let $A_X$ be any sheaf on $X$ considered in \eqref{nota:4} and \eqref{nota:6}, and let $\cA_\cX$ be the corresponding analytic sheaf on $\cX$. Then $\cA_\cX$ is $\pi^{-1}A_X$-flat. Furthermore, an~$A_X$\nobreakdash-mod\-ule $N$ has no $\CC[\hb]$-torsion (or $\CC[\hc]$-torsion) if and only if its analytification $\cN:=\cA_\cX\otimes_{\pi^{-1}A_X}\pi^{-1}N$ has no such torsion.

\item\label{nota:10}
Any holomorphic map $f:X\to Y$ between complex manifolds induces a holomorphic map between the corresponding associated spaces in \eqref{nota:2} and \eqref{nota:3}, that we still denote by $f$. For all sheaves considered in \eqref{nota:4}--\eqref{nota:6}, there is associated a transfer module, and the pullback and pushforward functors are defined correspondingly. For all these variants, we denote their $k$-th cohomological version by $\Dm f^{*(k)}$ and $\Dm f_*^{(k)}$ ($k\in\ZZ$). The context should make clear which category they apply to. When $f$ is flat, we simply denote $\Dm f^*$ instead of $\Dm f^{*(0)}$.
\end{enumerate}

\begin{remarque}
Proposition \ref{prop:subcoh} applies similarly in the setting of \eqref{nota:5}--\eqref{nota:6}.
\end{remarque}

\subsubsection*{Acknowledgements}
I would like to thank Takuro Mochizuki for pointing out a problem in the memoir \cite{Bibi15} and for developing an effective theory to improve the original (corrected) definition. I also thank him for the numerous discussions on the theory of mixed Hodge modules and for the interest he shows in the subject of irregular Hodge structures. I would also like to thank Christian Schnell and Jeng-Daw Yu for numerous enlightening discussions.

\section{A general framework for producing an irregular Hodge filtration}
\label{sec:rescalingalgebraic}

\subsection{A reminder on strict \texorpdfstring{$\RR$}{R}-specializability}\label{subsec:reminderV}
Let us assume that $X=Y\times\Delta_t$ (\cf Notation \ref{subsec:nota}\eqref{nota:3b}). Let $M$ be a coherent $R_X$\nobreakdash-mod\-ule. We say that $M$ is \emph{\sRspe along~$t$} if there exists an increasing\footnote{One often uses a decreasing $V$-filtration $V^\cbbullet M$; the correspondence is $V^aM=V_{-a}M$.} filtration $V_\bbullet M$ of $M$ indexed by $A+\ZZ$ for some discrete set $A\subset[0,1)$ satisfying the following properties:
\begin{itemize}
\item
for each $a\in A+\ZZ$, $V_a M$ is $V_0R_X$-coherent and $M=\bigcup_aV_a M$,
\item
$V_{\alpha+k}M=t^{-k}V_\alpha M$ for each $\alpha\in A$ and $k\in\ZZ_{\leq0}$,
\item
$V_aM=V_{{<}a}M+\partiall_tV_{a-1}M$ if $a>1$,
\item
$t\partiall_t+a\hb$ is nilpotent on $\gr_a^{V}M$ for any $a\in A+\ZZ$,
\item
for any compact set $K\subset X$, there exists a finite subset $A_K\subset A$ such that $\gr_a^{V}M|_K=0$ for $a \notin A_K+\ZZ$,
\item
multiplication by $\hb$ is injective on $\gr_a^{V}M$ for any $a\in A+\ZZ$ (we also say that $\gr_a^{V}M$ has no $\hb$-torsion).
\end{itemize}

Let us recall some properties of \sRspe modules.

\begin{itemize}
\item
If a $V$-filtration along $t$ exists for $M$, it is unique, so that checking \sRspy is a local question.
\item
Any morphism between \sRspe modules is (possibly not strictly) compatible with the $V$-filtration.
\end{itemize}

A similar definition exists for a coherent $\cR_\cX$-module $\cM$.

\begin{lemme}\label{lem:VVan}
Let $M$ be a coherent $R_X$-module and set $\cM=\cR_\cX\otimes_{\pi^{-1}R_X}\pi^{-1}M$. If $M$ is \sRspe along $t$, then $\cM$ is so, and the $V$-filtrations along $t$ satisfy $V_a\cM=V_0(\cR_\cX)\otimes_{\pi^{-1}V_0(R_X)}\pi^{-1}V_aM$ for each $a\in A+\ZZ$.
\end{lemme}

\begin{proof}
Let us assume $M$ is \sRspe along $t$ and let us set $U_a\cM=V_0(\cR_\cX)\otimes_{\pi^{-1}V_0(R_X)}\pi^{-1}V_aM$ for each $a\in A+\ZZ$. Then $U_\bbullet\cM$ satisfies the characteristic properties of the $V$-filtration of $\cM$, since $V_0(\cR_\cX)$ is $\pi^{-1}V_0(R_X)$-flat, hence is equal to it by uniqueness.
\end{proof}

\Subsection{Rescaling}\label{subsec:rescalingalgebraic}

\begin{definition}[Rescaling in the $\hb$-algebraic setting]\label{def:rescalingalgebraic}
Given a left $\wt R_X$-module $M$, the module $\tauM:=\ccO_X[\tau,\taum,\hc]\otimes_{\ccO_X[\hb]}M$ is naturally equipped with the structure of an $\wt \tauR_X(*\tau)$-module as follows:
\begin{enumeratea}
\item\label{enum:a}
We identify $\tauM$ with $\ccO_X[\tau,\taum]\otimes_{\ccO_X}M$ as an $\ccO_X[\tau,\taum]$-module.
\item
The action of $\hc$ is defined by $\hc\cdot\bigoplus_{k\in\ZZ}(\tau^k\otimes m_k)=\bigoplus_{k\in\ZZ}\tau^{k+1}\otimes(\hb m_k)$.
\item
The action of derivations is defined by (with the same notation $\partiall_{x_i}$ for the distinct elements $\hb\partial_{x_i}$ of $\wt R_X$ and $\hc\partial_{x_i}$ of $\wt \tauR_X(*\tau)$):
\begin{equation*}
\begin{aligned}
\partiall_{x_i}(1\otimes m)&=\tau(1\otimes\partiall_{x_i}m)=\tau\otimes\partiall_{x_i}m,\\
\partiall_\tau(1\otimes m)&=-1\otimes\hb^2\partial_\hb m,\\
\hc^2\partial_\hc(1\otimes m)&=\tau(1\otimes\hb^2\partial_\hb m)=\tau\otimes\hb^2\partial_\hb m,
\end{aligned}
\end{equation*}
and extended in a natural way to $\tauM$ by means of the decomposition \eqref{enum:a}.
\end{enumeratea}
\end{definition}

The identification \ref{def:rescalingalgebraic}\eqref{enum:a} provides $\tauM$ with a natural grading by the degree in $\tau$: $\tauM=\bigoplus_{k\in\ZZ}\tau^k\otimes M$. With respect to it, the sections of $\ccO_X[z]$ act in an homogeneous way with degree zero, and~$\hc,\tau$ with degree one. The derivations $\partiall_{x_i}$, $\tau\partiall_\tau$ and $\hc^2\partial_\hc$ are also homogeneous of \hbox{degree} one and one can identify the homogeneous component $\tau^k\otimes M$ with \hbox{$\ker(\hc^2\partial_\hc+\tau\partiall_\tau-k\hc)$}. We say that $\tauM$ is \emph{strict} if $\tauM$ has no $\CC[\hc]$-torsion, and in particular $\tauM\subset\tauM(*\hc)$.

Any $\wt R_X$-linear morphism $\lambda{:}\,M_1\!\to\! M_2$ yields a $\wt\tauR_X$-linear morphism \hbox{$\taulambda{:}\,\tauM_1\!\to\!\tauM_2$}, and $\wt\tauR_X$-linearity implies that $\taulambda$ is graded, since it commutes with $\hc^2\partial_\hc+\tau\partiall_\tau$.

If the $\wt R_X$-module $M$ is coherent as an $R_X$-module, it is $\wt R_X$\nobreakdash-co\-herent and its localized module $M(*\hb):=\wt R_X(*\hb)\otimes_{\wt R_X}M$ is a coherent $\cD_X[\hb,\hbm]\langle\partial_\hb\rangle$-module (\cf Notation \ref{subsec:nota}\eqref{nota:1}). We say that $M$ is \emph{strict} if it has no $\CC[\hb]$-torsion, and in particular $M\subset M(*\hb)$. We denote by $\ccM$ the $\ccD_X$-module $M/(\hb-1)M$.

\begin{proposition}[Coherence]
Assume that the $\wt R_X$-module $M$ is $R_X$-coherent. Then the $\wt\tauR_X(*\tau)$-module $\tauM$ is $\tauR'_X(*\tau)$-coherent, hence also $\wt\tauR_X(*\tau)$-coherent.
\end{proposition}

\begin{proof}
If $M$ is any $R_X$-module, the rescaled object $\tauM$ is well-defined as a $\tauR'_X$\nobreakdash-mod\-ule, and if $M$ is any $\wt R_X$-module, this structure is enhanced to a $\wt\tauR_X$-structure by Definition~\ref{def:rescalingalgebraic}. In~par\-ticular, ${}^\tau\!(R_X)\simeq\tauR'_X(*\tau)$ by sending $1\otimes\partiall_{x_i}$ to $\taum\partiall_{x_i}$. The assertion is then clear.
\end{proof}

\begin{lemme}\label{lem:tauMstrict}
If the $\wt R_X$-module $M$ is strict, then $\tauM$ is strict and $(\hc-\tau)$ acts in an injective way on $\tauM$. Furthermore, the restricted module \hbox{$i^*_{\tau=\hc}\tauM:=\tauM/(\hc-\tau)\tauM$} is naturally identified, as a graded module, to $\CC[\tau,\taum]\otimes_\CC\ccM$ with its natural grading induced by that of $\CC[\tau,\taum]$. If~$M$ is $R_X$-coherent, then $i^*_{\tau=\hc}\tauM$ is $\ccD_X[\tau,\taum]$-coherent.
\end{lemme}

\begin{proof}
We use the decomposition \ref{def:rescalingalgebraic}\eqref{enum:a} above. For strictness, let $p(\hc)=\sum_{i=\delta}^da_i\hc^i\in\CC[\hc]$ with $d\geq1$ and $a_\delta,a_d\neq0$. Let $m\in\tauM$ be nonzero written as $m=\bigoplus_{k=k_0}^{k_1}\tau^k\otimes m_k$ with $m_{k_0}\neq0$, such that $p(\hc)m=0$. This implies that, for each $k\in\ZZ$, we have $\sum_{i=\delta}^da_i\hb^im_{k-i}=0$. Taking $k=k_0+\delta$ yields $a_\delta\hb^\delta m_{k_0}=0$, in~contradiction with the strictness assumption for $M$.

For the injectivity of $(\hc-\tau)$, we have
\begin{equation}\label{eq:hbtau}
(\hc-\tau)\cdot\bigoplus_{k\in\ZZ}(\tau^k\otimes m_k)=\bigoplus_{k\in\ZZ}\tau^{k+1}\otimes(\hb-1) m_k,
\end{equation}
and the assertion follows from the strictness of $M$.

The quotient $\tauM/(\hc-\tau)\cdot\tauM$ is identified with $\bigoplus_k\tau^k\otimes_\CC\ccM$ as a graded module. The action induced by that of $\partiall_{x_i}$ is by $\tau\otimes\partial_{x_i}$. The last assertion is then clear.
\end{proof}

\begin{proposition}\label{prop:basic}
Let $M$ be an $\wt R_X$-module which is $R_X$-coherent and strict. Then, for any coherent $\tau$-graded $\tauR'_X$-submodule $\tauM_0$ of $\tauM$ which satisfies
\begin{starequation}\label{eq:basicstar}
(\tau-\hc)\cdot\tauM_0=\tauM_0\cap(\tau-\hc)\cdot\tauM
\end{starequation}%
and
\begin{starstarequation}\label{eq:basicstarstar}
\ccO_X[\tau,\taum,\hc]\otimes_{\ccO_X[\tau,\hc]}\tauM_0=\tauM,
\end{starstarequation}%
there exists a unique coherent filtration $F_\bbullet\ccM$ of the $\ccD_X$-module $\ccM=M/(\hb-1)M$ such that
\[
R_F\ccM:=\bigoplus_k\tau^k\otimes F_k\ccM=i^*_{\tau=\hc}\tauM_0\subset i^*_{\tau=\hc}\tauM=\CC[\tau,\taum]\otimes_\CC\ccM.
\]
\end{proposition}

\begin{proof}
We identify $i^*_{\tau=\hc}\tauR'_X$ with $R_F\ccD_X=\bigoplus_k\tau^k\otimes F_k\ccD_X$. Since $\tauM_0$ is $\tau$-graded, that is, $\tauM_0=\bigoplus_k\tau^k\otimes \tauM_0^{(k)}$ with $\tauM_0^{(k)}\subset M$, and since $\tau-\hc$ is homogeneous of degree one, the pullback $i^*_{\tau=\hc}\tauM_0$ is also $\tau$\nobreakdash-graded. Moreover, \eqref{eq:basicstar} implies that it is contained in $i^*_{\tau=\hc}\tauM=\CC[\tau,\taum]\otimes_\CC\nobreak\ccM$, hence has no $\CC[\tau]$-torsion, and since $\tauM_0$ is $\tauR'_X$\nobreakdash-coherent, $i^*_{\tau=\hc}\tauM_0$ is a graded coherent $R_F\ccD_X$-module. Lastly, \eqref{eq:basicstarstar} implies that $\CC[\tau,\taum]\otimes_{\CC[\tau]}(i^*_{\tau=\hc}\tauM_0)=\CC[\tau,\taum]\otimes_\CC\ccM$.

The proposition follows from these properties.
\end{proof}

We now give a sufficient condition for the existence of $\tauM_0$ as in the proposition. For that purpose, let us recall the notion of $V$-filtration and \sRspy along $\tau$. We~consider the $V$-filtration of $\tauR_X(*\tau)$ along $\tau$ with
\[
\tauV_0(\tauR_X(*\tau))=\ccO_X[\tau,\hc]\langle\partiall_x,\tau\partiall_\tau\rangle=\tauR'_X\langle\tau\partial_\tau\rangle,
\]
and $\tauV_k(\tauR_X(*\tau))=\tau^{-k}\cdot\tauV_0(\tauR_X(*\tau))$ for $k\in\ZZ$ (\cf Notation \ref{subsec:nota}\eqref{nota:6}). We adapt below the definition of Section \ref{subsec:reminderV} to the localized case.

\begin{definition}\label{def:sRspyalg}
We say that the $\tauR_X(*\tau)$-module $\tauM$ is \emph{\sRspe along~$\tau$} if there exists an increasing filtration $\tauV_\bbullet(\tauM)$ of $\tauM$ indexed by $A+\ZZ$ for some discrete set $A\subset[0,1)$ satisfying the following properties:
\begin{itemize}
\item
for each $a\in A+\ZZ$, $\tauV_a(\tauM)$ is $\tauV_0(\tauR_X(*\tau))$-coherent and $\tauM=\bigcup_a\tauV_a(\tauM)$,
\item
$\tauV_{a+k}(\tauM)=\tau^{-k}\tauV_a(\tauM)$ for each $a\in A+\ZZ$ and $k\in\ZZ$,
\item
$\tau\partiall_\tau+a\hc$ is nilpotent on $\gr_a^{\tauV}(\tauM)$ for any $a\in A+\ZZ$,
\item
for any compact set $K\subset X$, there exists a finite subset $A_K\subset A$ such that $\gr_a^{\tauV}(\tauM)|_K=0$ for $a \notin A_K+\ZZ$,
\item
multiplication by $\hc$ is injective on $\gr_a^{\tauV}(\tauM)$ for any $a\in A+\ZZ$ (we also say that $\gr_a^{\tauV}(\tauM)$ has no $\hc$-torsion).
\end{itemize}
\end{definition}

As in the case of Section \ref{subsec:reminderV}, the following holds.

\begin{itemize}
\item
If a $\tauV$-filtration exists for $\tauM$, it is unique.
\item
Any morphism between \sRspe $\tauR_X(*\tau)$-modules is compatible with the $\tauV$-filtration.
\item
We say that a morphism $\varphi:\tauM_1\to\tauM_2$ is \sRspe along $\tau$ if, for each $a\in A+\ZZ$ (index set for both $\tauV_\bbullet(\tauM)_1$ and $\tauV_\bbullet(\tauM)_2$), $\coker\gr_a^{\tauV}\varphi$ has no $\CC[\hc]$-torsion. If $\varphi$ is \sRspe along $\tau$, then the $\tauR_X(*\tau)$-modules $\ker\varphi$, $\image\varphi$ and $\coker\varphi$ are \sRspe along $\tau$ and their $\tauV$-filtration is the one naturally induced by that of $\tauM_1$ or $\tauM_2$.
\end{itemize}

\begin{proposition}\mbox{}\label{prop:rescalV}
\begin{enumerate}
\item\label{prop:rescalV1}
Assume that the $\wt R_X$-module $M$ is $R_X$-coherent and strict and that $\tauM$, as~a $\tauR_X(*\tau)$\nobreakdash-mod\-ule, is \sRspe along $\tau$. Let $\tauV_\bbullet(\tauM)$ denote the corresponding $V$-filtration. Then each $\tauV_a(\tauM)$ is a coherent $\tau$-graded $\tauR'_X\langle\tau\partial_\tau\rangle$-module which satisfies \eqref{eq:basicstar} and \eqref{eq:basicstarstar}.

\item\label{prop:rescalV3}
For $M$ as in \eqref{prop:rescalV1}, if furthermore for some $a\!\in\!\RR$, $\tauV_a(\tauM)$ is coherent over~$\tauR'_X$, then it is so for any $a\in\RR$ and each $\tauV_a(\tauM)$ gives rise to a $\ZZ$-indexed coherent filtration denoted by $F^\irr_{a,\bbullet}\ccM$. The notation is coherent in the sense that, for any $k\in\ZZ$, $\tauV_{a+k}(\tauM)$ gives rise to the $k$-shifted filtration, that is,
\begin{starequation}\label{eq:rescalV2star}
F^\irr_{a+k,\bbullet}\ccM=F^\irr_{a,\bbullet+k}\ccM,
\end{starequation}%
so that we denote it as $F^\irr_{a+\bbullet}\ccM$. Moreover, for any $a,b\in\RR$, we have the filtration property
\begin{starstarequation}\label{eq:rescalV2starstar}
a\leq b\implies F^\irr_{a}\ccM\subset F^\irr_{b}\ccM,
\end{starstarequation}%
so that the family of $\ZZ$-indexed filtrations $(F^\irr_{\alpha+\bbullet}\ccM)_{\alpha\in A}$ forms an $\RR$-indexed filtration $F^\irr_\bbullet\ccM$.
\item
For $M_1,M_2$ as in \eqref{prop:rescalV3}, and any $\wt R_X$-linear morphism $\lambda:M_1\to M_2$, let \hbox{$\taulambda:\tauM_1\to\tauM_2$} be the associated $\wt\tauR_X$-linear morphism. If $\lambda$ is strict and $\taulambda$ is \sRspe along $\tau$, then the $\ccD_X$-linear morphism $\lambda|_{\hb=1}:\ccM_1\to\ccM_2$ is strictly filtered with respect to the $\RR$-indexed filtration $F^\irr_{\cbbullet}\ccM$.
\end{enumerate}
\end{proposition}

\pagebreak[2]
\begin{proof}\mbox{}
\begin{enumerate}
\item
Recall that $\tau^k\tauV_a(\tauM)=\tauV_{a-k}(\tauM)$ for each $a\in\RR$ and $k\in\ZZ$. Since the $V$\nobreakdash-fil\-tra\-tion is exhaustive, \eqref{eq:basicstarstar} holds.

For \eqref{eq:basicstar}, let $m$ be a local section of $\tauV_a(\tauM)$ such that $m=(\tau-\hc)m'$, where~$m'$ is a local section of $\tauV_b(\tauM)$ for some $b>a$. Then the class of $(\tau-\hc)m'$ in $\gr^{\tauV}_b(\tauM)$ is zero. This is nothing but the class of $-\hc m'$. The \sRspy property implies that $\hc$ acts in an injective way on each $\gr^{\tauV}_b(\tauM)$. Therefore, $m'$ is a local section of $\tauV_{<b}(\tauM)$. Arguing inductively, one eventually finds that $m'$ is a local section of $\tauV_{a}(\tauM)$.

We are left with proving that each $\tauV_{a}(\tauM)$ is $\tau$-graded. This property is equivalent to the fact that each $\tauV_{a}(\tauM)$ is preserved by the $\CC^*$-action on $\tauM$ attached to the grading \ref{def:rescalingalgebraic}\eqref{enum:a}: for any $c\in\CC^*$, we have by definition $c\cdot\bigl(\bigoplus_k\tau^k\otimes m_k\bigr)=\bigoplus_k\tau^k\otimes(c^km_k)$. Since the action of $\ccO_X$ is homogeneous of degree zero and $\tau,\hc,\tau\partiall_\tau,\partiall_x$ are homo\-geneous of degree one, we deduce that, for each $c\in\CC^*$, the filtration $c\cdot\nobreak\tauV_{\bbullet}(\tauM)$ satisfies all properties of the $V$-filtration. Since the $V$-filtration is unique, we conclude that $\tauV_{\bbullet}(\tauM)$ is $\tau$-graded. (We owe this argument to T.\,Mochizuki \cite[Lem.\,6.10]{Mochizuki21}.)

\item
With the supplementary coherency assumption, we can apply Proposition \ref{prop:basic} to each $\tauV_{a}(\tauM)$ and obtain the coherent filtration $F^\irr_{a,\bbullet}\ccM$. Then \eqref{eq:rescalV2star} follows from the equality $\tau^k\tauV_a(\tauM)=\tauV_{a-k}(\tauM)$. Similarly, \eqref{eq:rescalV2starstar} follows from the inclusion $\tauV_a(\tauM)\subset\tauV_b(\tauM)$ if $a\leq b$.\enlargethispage{.5\baselineskip}%

\item
It follows from the assumption of \eqref{prop:rescalV3} on $M_1,M_2$ that, for $C=\ker,\image,\coker$, $C(\taulambda)$ also satisfies it. Moreover, we also find that $C(i^*_{\tau=\hc}\taulambda)=i^*_{\tau=\hc}C(\taulambda)$. We~conclude that for each $a\in\RR$, $i^*_{\tau=\hc}\taulambda:R_{F^\irr_\alpha}M_1\to R_{F^\irr_\alpha}M_2$ is a strict morphism, \ie for each $a\in\RR$ and $k\in\ZZ$, $\lambda|_{\hb=1}(F^\irr_{a+k}M_1)=F^\irr_{a+k}M_2\cap\image\lambda|_{\hb=1}$. This is the desired strictness property.
\qedhere
\end{enumerate}
\end{proof}

\begin{remarque}
Assume that $M$ satisfies the properties of Proposition \ref{prop:rescalV}\eqref{prop:rescalV1}. Let $G:\tauM\to\tauM$ denote the grading operator, that is, $G\bigl(\bigoplus_k\tau^k\otimes m_k\bigr)=\bigoplus_k\tau^k\otimes km_k$. Then each $\tauV_a(\tauM)$ is preserved by $G$, according to Proposition \ref{prop:rescalV}\eqref{prop:rescalV1}, and also by~$\tau G$. On the other hand, $\hc^2\partial_\hc$ acts on $\tauM$ as $\tau G-\tau\partiall_\tau$. It follows that each $\tauV_a(\tauM)$ is acted on by $\hc^2\partial_\hc$, hence is a $V_0\wt\tauR_X$-module.
\end{remarque}

\subsection{Provisional corrected definition of \texorpdfstring{$\IrrMHM$}{IrrMHM}}\label{subsec:IrrMHMprovisional}

We refer to the notations of \cite{Bibi15}. As in \loccit, we replace the category $\MTM^\intt(X)$ of \cite{Mochizuki11} with the category $\iMTM^\intt(X)$, where we have modified the notion of sesquilinear pairing to enable rescaling. Let $X$ be a complex manifold. We now define the category $\IrrMHM (X)$ as a full subcategory of the category of $W$-filtered triples $((M',M'',\iC),W_\bbullet)$, where $M',M''$ are $\wt R_X$-modules which are $R_X$-coherent, $\iC$ is a $\iota$-sesquilinear pairing between the analytifications $\cM^\circ=\wt \cR_{\cX^\circ}\otimes_{\pi^{-1}\wt R_X}\pi^{-1}M$ ($M=M',M''$) in the sense of \cite[\S1.3.e]{Bibi15}, and $W_\bbullet$ is a filtration in such a category. (Morphisms are the usual ones.)

Such a filtered triple can be rescaled: for the components $M',M''$, this is Definition~\ref{def:rescalingalgebraic}, and for the sesquilinear pairing $\iC$, this is defined in \cite[\S2.2.c]{Bibi15}. We~denote the rescaled object by ${}^\tau((M',M'',\iC),W_\bbullet)$. The objects of $\IrrMHM (X)$ are those $W$-filtered triples $((M',M'',\iC),W_\bbullet)$ such that
\begin{enumerate}
\item\label{enum:oldIrrMHM1}
the analytifications $((\taucM',\taucM'',\tauiC),W_\bbullet)$ of ${}^\tau((M',M'',\iC),W_\bbullet)$ is an object of the category $\iMTM^\intt(\tauX,(*0_\tau))$,
\item\label{enum:oldIrrMHM2}
the $\cR_{\taucX}(*0_\tau)$-modules $\taucM',\taucM''$, which are strictly $\RR$-specializable along~$\tau$ because of \eqref{enum:oldIrrMHM1}, are also regular along $\tau$ and graded (\cf\cite[Def.\,2.19 \& Def.\,2.26]{Bibi15}).
\end{enumerate}

With this definition of $\IrrMHM(X)$, \ie assuming that the $W$-filtered triples we start with are algebraic in the $\hb$-direction, the results of \cite{Bibi15} hold true, except that it is mentioned in \cite[Th.\,0.2]{Bibi15} that $\IrrMHM(X)$ is a \emph{full} abelian subcategory of $\iMTM^\intt(X)$. Fullness is no longer true since morphisms should have components which are morphisms of $\wt R_X$-modules, while in $\iMTM^\intt(X)$ the components are \hbox{morphisms} of $\wt\cR_\cX$-modules. We will recover fullness with the improved definition of Section \ref{subsec:improved}.

The examples considered in \loccit\ all come with a natural partial $\hb$-algebraic structure. This directly follows from one of the results of \cite{Mochizuki21} (\cf Corollary \ref{cor:Malgrangeext} below). For the sake of completeness, we give a direct proof of the following lemma in order to ensure that the results of \cite{Bibi15} do not depend on those of \cite{Mochizuki21}, although the latter results enable us to give a better definition of $\IrrMHM(X)$ in Section \ref{subsec:improved}.

\begin{lemme}\label{lem:partialalg}
Let $\varphi$ be a meromorphic function on $X$ and let $\cT$ be the mixed twistor D-module associated to a mixed Hodge module. Then the $\wt\cR_\cX$-module underlying the integrable twistor D-module $\cT^{\varphi/\hb}\otimes\cT$ (\cf\cite[\S1.6.a]{Bibi15}) has a natural partial $\hb$-algebraic structure.
\end{lemme}

\begin{proof}[Sketch of proof]
The assertion is clear for $\cT$ alone: by definition, for a filtered $\cD_X$\nobreakdash-mod\-ule $(\ccM,F_\bbullet\ccM)$ as the one underlying a mixed Hodge module, the associated $R_X$\nobreakdash-mod\-ule is the Rees module $R_F\ccM$. It is thus enough to check the assertion for~$\cT^{\varphi/\hb}$. Let $P$ be the pole divisor of $\varphi$ and let $e:X'\to X$ be a projective modification which is an isomorphism over $U:=X\moins P$ and such that
\begin{itemize}
\item
$X'$ is smooth and $e^{-1}(P)$ is a divisor with normal crossings,
\item
$\varphi:=\varphi\circ e$ extends as a projective morphism $\varphi':X'\to\PP^1$,
\item
$e^{-1}(P)$ decomposes as $e^{-1}(P)=P'\cup H'$, where $P'$ denotes the pole divisor of~$\varphi'$.
\end{itemize}
In the setting of $\cR_{\cX'}$-modules, we set $\cE^{\varphi'/\hb}=(\cO_{\cX'}(*P'),\hb\rd+\rd\varphi')$ and we denote by $\cE^{\varphi'/\hb}[*H']$ the twistor localization of $\cE^{\varphi'/\hb}$ along $H'$, which underlies the mixed twistor D-module $\cT^{\varphi'/\hb}[*H']$ (\cf\cite[Prop.\,11.2.2]{Mochizuki11}, \cf also the proof of \cite[Prop.\,3.3]{S-Y14}). A local computation, using the fact that the variables involved in $\varphi'$ and in $H'$ can be made distinct, shows that $\cE^{\varphi'/\hb}[*H']\simeq\cE^{\varphi'/\hb}\otimes_{\cO_\cX}(\cO_\cX,\rd)[*H']$.

On the one hand, $\cE^{\varphi'/\hb}$ has the natural $\hb$-algebraic structure $(\cO_{X'}[\hb](*P'),\hb\rd+\rd\varphi')$ and, on the other hand, $(\cO_\cX,\rd)[*H']$ has a natural $\hb$-algebraic structure since it is the analytification of the Rees module of a filtered $\cD_{X'}$-module underlying a mixed Hodge module. Therefore, the tensor product $\cE^{\varphi'/\hb}[*H']$ of both also comes equipped with a natural algebraic structure. Last, the algebraic structure of $\cE^{\varphi/\hb}$ is obtained by pushforward by $e$ of the latter.
\end{proof}

\section{Improved definition of \texorpdfstring{$\IrrMHM$}{IrrMHM}}
Although the setting of Section \ref{subsec:rescalingalgebraic} leads in a direct way to the construction of the irregular Hodge filtration, the objects of $\IrrMHM(X)$ are analytic objects, and in order to improve the definition of $\IrrMHM(X)$ we have to take care that some properties needed in Proposition \ref{prop:rescalV} only hold in the algebraic setting. In this section, we~explain the various comparison results needed to adapt Proposition \ref{prop:rescalV} in the analytic setting before giving the improved definition of $\IrrMHM(X)$.

\Subsection{Goodness and extension to infinity of strict holonomic \texorpdfstring{$\wt\cR_\cX$}{wRX}-modules}
Let $\cM$ be a coherent $\cR_\cX$\nobreakdash-mod\-ule. We say that it is \emph{strict} if it has no $\cO_{\CC_\hb}$-torsion. We say that it is holonomic if its characteristic variety is contained in $\Lambda\times T^*\CC_\hb$ for some Lagrangian closed analytic set $\Lambda\subset T^*X$. For a coherent $\wt\cR_\cX$-module $\cM$, we~say that it is strict, \resp holonomic, if the underlying $\cR_\cX$-module is so.

\begin{proposition}[Goodness]\label{prop:holgood}
Let $\cM$ be a strict holonomic $\wt\cR_\cX$-module. Then $\cM$ is good. More precisely, $\cM$ is the union of an increasing sequence of coherent $\cO_{\cX}$\nobreakdash-sub\-modules (with no reference to a compact set).
\end{proposition}

\begin{lemme}\label{lem:holRD}
Let $\cM$ be a strict holonomic $\wt\cR_\cX$-module. Then the localized module $\cM(*0_\hb)$, regarded as a $\cD_\cX$-module, is holonomic.
\end{lemme}

\begin{proof}
Since we have the identification $\wt\cR_\cX(*0_\hb)=\cD_\cX(*0_\hb)$, we regard the localized module $\cM(*0_\hb):=\cO_\cX(*0_\hb)\otimes_{\cO_\cX}\cM$ as a $\cD_\cX(*0_\hb)$-module, hence as a $\cD_\cX$-module.

It has been observed in \cite[Prop.\,1.26]{Bibi15}, as a consequence of the main theorems of \cite{Kashiwara78}, that, as such, $\cM(*0_\hb)$ is $\cD_\cX$-holonomic.
\end{proof}

\begin{proof}[Proof of Proposition \ref{prop:holgood}]
Lemma \ref{lem:holRD}, together with \cite[Th.\,II.3.1]{Malgrange04}, implies that $\cM(*0_\hb)$ admits a coherent filtration, and in particular it is the union of an increasing sequence of coherent $\cO_{\cX}$-submodules $\cF_k$. By strictness, we have $\cM\subset\cM(*0_\hb)$. We~conclude with Proposition \ref{prop:subcoh}.
\end{proof}

\begin{proposition}\label{prop:fM}
Let $\cM$ be a strict holonomic $\wt\cR_\cX$-module. There is a one-to-one correspondence between the following objects:
\begin{enumerate}
\item\label{prop:fM1}
An $\wt\cR_\fX$-coherent extension $\fM$ of $\cM$ to $X\times\PP^1$ (\cf Notation \ref{subsec:nota}\eqref{nota:3}),
\item\label{prop:fM2}
a coherent $\wt R_X$-module $M$ such that $\cM=\wt\cR_\cX\otimes_{\pi^{-1}\wt R_X}\pi^{-1}M$.
\end{enumerate}
\end{proposition}

\begin{proof}
Assume \eqref{prop:fM2}. Clearly, given $M$ defining $\cM$, we can extend $\cM$ by setting
\[
\fM:=\wt\cR_\fX\otimes_{\pi^{-1}\wt R_X}\pi^{-1}M.
\]

On the other hand, given $\fM$ as in \eqref{prop:fM1}, $\fM(*0_\hb)$ is a coherent $\cD_{\fX}(*(0_\hb\cup\nobreak\infty_\hb))$-mod\-ule whose restriction to $\cX^\circ$ is holonomic. According to \cite[Th.\,1.2 \& 1.3]{Kashiwara78}, it is therefore $\cD_\fX$-holonomic and \cite[Th.\,II.3.1]{Malgrange04} ensures that it has a global coherent filtration by $\cO_\fX$-submodules. Since $\cM$ is strict, we have $\fM\subset\fM(*0_\hb)$ and the argument already used in Proposition \ref{prop:holgood} implies that $\fM$ also has a coherent filtration. Then the argument of \cite[Th.\,A.1]{D-S02a} extends to the present setting in order to construct $M$.
\end{proof}

\subsection{Results of T.\,Mochizuki in \cite{Mochizuki21}}\label{subsec:resultsMochizuki}

The first improvement obtained in \cite{Mochizuki21} enables us to avoid assuming a priori that the objects we start with are algebraic in the $\hb$-direction. When $X$ is reduced to a point, this was already done in \cite[Chap.\,3]{Bibi15} by using the notion of Deligne meromorphic extension. When $\dim X\geq1$, as we work with possibly irregular holonomic D-modules, the existence of such an extension is much less obvious, and is called the \emph{Malgrange extension} in \cite{Mochizuki21}.

We recall that an integrable mixed twistor D-module is a $W$-filtered triple $((\cM',\cM'',\iC),W_\bbullet)$, where $\cM',\cM''$ are strict holonomic $\wt\cR_\cX$-modules and $\iC$ is an integrable sesquilinear pairing (that can replace $C$ of \cite{Mochizuki11} in the integrable case, see \cite[Chap.\,1]{Bibi15}).

\begin{theoreme}[{\cite[Th.\,1.3]{Mochizuki21}}]\label{th:Malgrangeext}
Assume that $\cM=\cM',\cM''$ underlies an integrable mixed twistor D-module. Then there exists a unique Malgrange extension of $\cM$, that~is, a holonomic $\wt\cR_\fX$-module $\fM$ which is regular along $\infty_\hb$. This correspondence is functorial with respect to morphisms in $\iMTM^\intt(X)$.\qed
\end{theoreme}

\begin{corollaire}\label{cor:Malgrangeext}
Any integrable mixed twistor D-module $((\cM',\cM'',\iC),W_\bbullet)$ comes canonically, by analytification, from a $W$-filtered triple $((M',M'',\iC),W_\bbullet)$.
\end{corollaire}

\begin{proof}
This is a consequence of Theorem \ref{th:Malgrangeext} and Proposition \ref{prop:fM}.
\end{proof}

Next, in \cite[\S2.3.d]{Bibi15}, a \emph{well-rescalability} property property is assumed for the objects of $\IrrMHM(X)$. The following result of \cite{Mochizuki21} asserts that, by using the canonical $\hb$-algebraic representative for rescaling an object of $\iMTM^\intt(X)$, this property is automatically satisfied.

\begin{theoreme}[{\cite[Th.\,1.5]{Mochizuki21}}]\label{th:wellresc}
Assume that $\cM$ underlies an integrable mixed twistor D-module on $X$. Let $M$ be the $\hb$-algebraic representative of its Malgrange extension, let $\tauM$ be the rescaled object (\cf Definition\ref{def:rescalingalgebraic}) and let $\taucM$ denotes its analytification, which is an $\wt\cR_{\taucX}(*0_\tau)$-module. Assume furthermore that $\taucM$ is the naive localization along~$\tau$ of an $\wt\cR_{\taucX}$\nobreakdash-module which underlies a mixed twistor D-module on $\tauX$. Then~$\taucM$ is regular along $\tau$, \ie $\cM$ is well-rescalable in the sense of \cite[\S2.3.d]{Bibi15}.\qed
\end{theoreme}

\subsection{Specializability along \texorpdfstring{$\tau$}{tau0}}
In \loccit, a grading property (\cite[Def.\,2.26]{Bibi15}) is also imposed in the definition of $\IrrMHM$. It was noticed in \cite{Mochizuki21} that this property is automatically satisfied (invariance under a $\CC^*$-action). We explain this property in the setting considered above.

We take up the setting of Section \ref{sec:rescalingalgebraic}. Let $M$ be a \emph{strict} $\wt R_X$-module (\ie with no $\CC[\hb]$-torsion). We say that $M$ is a \emph{holonomic $\wt R_X$-module} if its analytification $\cM:=\wt\cR_\cX\otimes_{\pi^{-1}\wt R_X}\pi^{-1}M$ is so, that is, the characteristic variety of $\cM$ is contained in $\Lambda\times T^*\CC_\hb$ for some closed Lagrangian analytic subset~$\Lambda$ of $T^*X$. Let $\tauM$ be the rescaled object (\cf Definition~\ref{def:rescalingalgebraic}) and let $\taucM$ denotes its analytification on $\taucX$, which is an $\wt\cR_{\taucX}(*0_\tau)$-module. The notion of \sRspy along $\tau$ in this analytic setting is similar to that of Definition \ref{def:sRspyalg}, and the properties of the $\tauV$-filtration are similar. In this section we prove the next proposition.

\begin{proposition}\label{prop:graded}
Assume that $M$ is holonomic, strict, and that $\taucM$ is \sRspe along $\tau$. Then $\tauM$ is \sRspe along $\tau$ and we have, for each $a\in A+\ZZ$,
\[
\tauV_a(\taucM)=V_0(\cR_{\taucX})\otimes_{q^{-1}V_0(\tauR_X)}q^{-1}\tauV_a(\tauM),
\]
where $q$ is the projection $\taucX\to X$.
\end{proposition}

\begin{corollaire}\label{cor:graded}
Under the assumptions of Proposition \ref{prop:graded}, the $\cR_\cX$-module $\cM=\cR_\cX\otimes_{\pi^{-1}R_X}\pi^{-1}M$ is graded in the sense of \cite[Def.\,2.26]{Bibi15}.
\end{corollaire}

Note that here we do not take care of the condition of \emph{well}-rescalability.

\begin{proof}[Proof of Corollary \ref{cor:graded}]
Each $\tauV_a(\tauM)$ is $\tau$-graded, according to Proposition \ref{prop:rescalV}\eqref{prop:rescalV1}, so~grading $\tauV_a(\taucM)$ with respect to the $\tau$-adic filtration yields $\tauV_a(\tauM)$ and the identification of Proposition \ref{prop:graded}, after restricting it by the functor $i^*_{\tau=\hc}$, implies the grading property.
\end{proof}

\subsubsection*{Proof of Proposition \ref{prop:graded}}
If $M$ is holonomic and strict, then so is $\cM$, and its localized module $\cM(*0_\hb)$ is a holonomic $\cD_\cX$-module (Lemma \ref{lem:holRD}). Furthermore, the analytification $\fM(*0_\hb)$ on $\fX=X\times\PP^1_\hb$, which is a coherent $\cD_{\fX}(*(0_\hb\cup\infty_\hb))$-module, is~also $\cD_{\fX}$-holonomic, according to \cite{Kashiwara78}, and good according to \cite[Th.\,II.3.1]{Malgrange04}. By~apply\-ing the functor $\pi_*$ (\cf \cite[Th.\,A.1]{D-S02a}), we deduce that $M(*\hb)$ is $\ccD_X[\hb]\langle\partial_\hb\rangle$-holonomic (holonomicity means that $\mathrm{Ext}^k_{\ccD_X[\hb]\langle\partial_\hb\rangle}(M(*\hb),\ccD_X[\hb]\langle\partial_\hb\rangle)=0$ for $k\neq\dim X+1$), and also as a $\ccD_X[\hb,\hbm]\langle\hb\partial_\hb\rangle$-module. We will use similar arguments to obtain the $\tauV$-filtration of $\tauM$ from that of $\taucM$.

The rescaled module $\tauM$ is a $\ccD_X[\tau,\taum,\hc]\langle\partiall_\tau,\hc^2\partial_\hc\rangle$-module. We will analytify it as a sheaf $\taufM$ on $X\times\PP^1_\tau\times\PP^1_\hc$ over the coherent sheaf of rings $\wt\cR_{\ov\taufX}(*0_\tau)$ (\cf Notation~\ref{subsec:nota}\eqref{nota:5}).

The rescaled module ${}^\tau\!(M(*\hb))$ is a $\ccD_X[\tau,\taum,\hc,\hcm]\langle\tau\partial_\tau,\hc\partial_\hc\rangle$-module equal to $\tauM(*\hc)$ and is identified to the pullback $D$-module of $M(*\hb)$ by the map \hbox{$(\hc,\tau)\mto\hb=\hc\taum$}. In~par\-ticular it is also holonomic, and has a decomposition
\begin{equation}\label{eq:decomptauMhcm}
\tauM(*\hc)=\bigoplus_k(\tau^k\otimes M(*\hb)).
\end{equation}
Furthermore, since $\tau^k\otimes M(*\hb)$ is identified with $\ker(\hc^2\partial_\hc+\tau\partiall_\tau-k\hc)$, we deduce that $\tauM(*\hc)$ is $\tauR_X(*\hc)$-coherent (\ie the action of $\partial_\hc$ can be expressed in terms of that of $\hcm\tau\partial_\tau$). Its analytification is denoted by $\taufM(*0_\hc)$.

From \cite{Kashiwara78} we deduce:

\begin{corollaire}\mbox{}
\begin{enumerate}
\item
The analytification $\taufM(*0_\hc)$ on $X\times\PP^1_\tau\times\PP^1_\hc$ of $\tauM(*\hc)$ is a holonomic $\cD_{X\times\PP^1_\tau\times\PP^1_\hc}$-module which is equal to its localization at $0_\tau\cup\infty_\tau\cup0_\hc\cup\infty_\hc$.
\item
The restriction $\taucM(*0_\hc)$ of $\taufM(*0_\hc)$ to $\taucX$ is a holonomic $\cD_{\taucX}$-module equal to its localization along $0_\tau\cup0_\hc$. It is also equal to the analytification of $\tauM(*\hc)$ on~$\taucX$.
\qed
\end{enumerate}
\end{corollaire}

Note also that the restriction of $\taucM(*0_\hc)$ or $\taufM(*0_\hc)$ to the open subset $\taucX^\circ$ is equal to the restriction of $\taucM$ or $\taufM$ to~$\taucX^\circ$. Under the assumptions of Proposition~\ref{prop:graded}, we have the following properties.
\begin{itemize}
\item
As $\taufM(*0_\hc)$ is $\cD_{X\times\PP^1_\tau\times\PP^1_\hc}$-holonomic, it admits a Kashiwara-Malgrange filtration $\tauV_\bbullet(\taufM(*0_\hc))$ along $\tau$. This filtration is indexed by a set $A'+\ZZ$, where $A'$ is a priori only a discrete subset of $\CC$ that we can choose in $\{a\in\CC\mid\reel(a)\in[0,1)\}$ with the total order induced by the lexicographic order on $\CC=\RR+\sfi\RR$. For any compact set $K\subset X$, when restricting to $K\times\PP^1_\tau\times\PP^1_\hc$, only a finite subset \hbox{$A'_K\subset A'$} occurs as the index set. Furthermore, away from $0_\tau$, each $\tauV_a(\taufM(*0_\hc))$ \hbox{coincides}~with~$\taufM(*0_\hc)$ and, by~\hbox{restricting} to $\taucX^\circ$ (\ie away from $\infty_\tau\cup0_\hc\cup\infty_\hc$), we find the Kashiwara-Malgrange filtration as holonomic $\cD_{\taucX^\circ}$-module
\[
\tauV_a(\taufM(*0_\hc))|_{\taucX^\circ}=\tauV_a(\taufM(*0_\hc)|_{\taucX^\circ})=\tauV_a(\taucM(*0_\hc)|_{\taucX^\circ}).
\]

\item
On the other hand, by assumption of \sRspy of $\taucM$ along $\tau$, $\taucM$ is equipped with a coherent $\tauV$-filtration with respect to $\tauV_\bbullet(\cR_{\taucX})$.
\end{itemize}

As a consequence, $\taucM|_{\taucX^\circ}=\taufM(*0_\hc)|_{\taucX^\circ}$ comes the equipped with two a~priori distinct $\tauV$-filtrations (\ie with respect to distinct rings of operators). Indeed, we~regard $\cR_{\taucX^\circ}$ as a subsheaf of rings of $\wt\cR_{\taucX^\circ}=\cR_{\taucX^\circ}\langle\partial_\hc\rangle=\cD_{\taucX^\circ}$. Although $\taucM|_{\taucX^\circ}$ is $\cR_{\taucX^\circ}$-coherent, we cannot assert a priori that its $\tauV_\bbullet(\cD_{\taucX^\circ})$-filtration (which exists by the holonomicity assumption) and its $\tauV_\bbullet(\cR_{\taucX^\circ})$-filtration (which exists by assumption) coincide. This would be the case if we would know that each term of its $\tauV_\bbullet(\cD_{\taucX^\circ})$-filtration is $\tauV_0(\cR_{\taucX^\circ})$-coherent, by uniqueness of the $\tauV_\bbullet(\cR_{\taucX^\circ})$-filtration.

\begin{lemme}
Each term of the $\tauV_\bbullet(\cD_{\taucX^\circ})$-filtration of $\taucM(*0_\hc)|_{\taucX^\circ}$ is $\tauV_0(\cR_{\taucX^\circ})$-coherent.
\end{lemme}

\begin{proof}[Sketch of proof]
The $\tauV_\bbullet(\cD_{\taucX^\circ})$-filtration of $\taucM(*0_\hc)|_{\taucX^\circ}$ is the restriction of the $\tauV_\bbullet(\cD_{\taucX^\circ})$-filtration of $\taufM(*0_\hc)$. Since each step of this filtration is good, we can apply the equivalence by $q_*$ (\cf \cite[Th.\,A.1]{D-S02a}) to show that the $\tauV_\bbullet(\cD_{\taucX^\circ})$-filtration of $\taucM(*0_\hc)|_{\taucX^\circ}$ comes by analytification from a $\tauV$-filtration of $\tauM(*\hc)$ as a $\cD_X[\tau,\taum,\hc,\hcm]\langle\tau\partial_\tau,\hc\partial_\hc\rangle$-module. It is the unique increasing filtration $\tauV_\bbullet(\tauM(*\hc))$ of $\tauM(*\hc)$ indexed by $A'+\ZZ$ for some discrete set $A'\subset\{a\in\CC\mid\reel(a)\in[0,1)\}$ satisfying the following properties:
\begin{itemize}
\item
for each $a\in A'+\ZZ$, $\tauV_a(\tauM(*\hc))$ is $\cD_X[\tau,\hc,\hcm]\langle\tau\partial_\tau,\hc\partial_\hc\rangle$-coherent and $\tauM(*\hc)=\bigcup_a\tauV_a(\tauM(*\hc))$,
\item
$\tauV_{a+k}(\tauM(*\hc))=\tau^{-k}\tauV_a(\tauM(*\hc))$ for each $a\in A'+\ZZ$ and $k\in\ZZ$,
\item
$\tau\partial_\tau+a$ is nilpotent on $\gr_a^{\tauV}(\tauM(*\hc))$ for any $a\in A'+\ZZ$,
\item
for any compact set $K\subset X$, there exists a finite subset $A'_K\subset A'$ such that $\gr_a^{\tauV}(\tauM(*\hc))|_K=0$ for $a \notin A'_K+\ZZ$.
\end{itemize}
We can now argue as in Proposition \ref{prop:rescalV}\eqref{prop:rescalV1}. By uniqueness of the $\tauV$-filtration, $\tauV_a(\tauM(*\hc))$ is stable by the $\CC^*$-action, and is thus graded with respect to the grading \eqref{eq:decomptauMhcm}. It follows that the action of $\partial_\hc$ can be expressed in terms of that of $\hcm\tau\partial_\tau$, and thus each $\tauV_a(\tauM(*\hc))$ is $\cD_X[\tau,\hc,\hcm]\langle\tau\partial_\tau\rangle$-coherent. By analytification, we obtain the desired $\tauV_0(\cR_{\taucX^\circ})$-coherency.
\end{proof}

Now that we know that the filtrations $\tauV_\bbullet(\taucM(*0_\hc)|_{\taucX^\circ})$ and $\tauV_\bbullet(\taucM)|_{\taucX^\circ}$ coincide, we can glue $\tauV_\bbullet(\taucM)$ with $\tauV_\bbullet(\taufM(*0_\hc))$ and for each $a$ we obtain a coherent sheaf over the ring $\tauV_0\cR_{\ov\taufX}$, which is good (being contained in a good $\cD_{\ov\taufX}(*(\infty_\tau\cup\infty_\hc))$-module, \cf Proposition \ref{prop:subcoh}). By taking pushforward by $q_*$ we obtain the filtration $\tauV_\bbullet(\tauM)$ which satisfies the properties of Proposition \ref{prop:graded}.\qed

\subsection{New definition of \texorpdfstring{$\IrrMHM$}{IrrMHM}}\label{subsec:improved}
We can now define $\IrrMHM(X)$ as a full subcategory of $\iMTM^\intt(X)$. Let $((\cM',\cM'',\iC),W_\bbullet)$ be an object of $\iMTM^\intt(X)$: in~par\-tic\-ular, it is a $W$-filtered triple consisting of $\wt\cR_\cX$-modules $\cM',\cM''$ which are holonomic and strict as $\cR_\cX$-modules, and of an integrable sesquilinear pairing $\iC$ between them. We say that $((\cM',\cM'',\iC),W_\bbullet)$ is an object of $\IrrMHM(X)$ if the following holds:
\begin{itemize}
\item
Let $((M',M'',\iC),W_\bbullet)$ be the partial algebraization of the Malgrange extensions of $((\cM',\cM'',\iC),W_\bbullet)$ (Corollary \ref{cor:Malgrangeext}), so that the rescalings $\taucM',\taucM''$ are well-defined $\wt\cR_{\taucX}(*0_\tau)$-modules by Theorem \ref{th:wellresc} and $\tauiC$ is defined as a sesquilinear pairing between their restriction to $X\times\CC^*_\tau\times\CC^*_\hc$. \emph{Then $\tauiC$ extends as a sesquilinear pairing between $\taucM'$ and $\taucM''$ with values in moderate distributions (\cf\cite[Def.\,2.43]{Bibi15}), and the object $((\taucM',\taucM'',\tauiC),W_\bbullet)$ belongs to $\MTM(\tauX,(*0_\tau))$.}
\end{itemize}

It follows from Theorem \ref{th:wellresc} and Proposition \ref{prop:graded} that the rescaling of an object of $\IrrMHM(X)$ is well-rescalable and graded in the sense of \cite[\S2.3.d \& Def.\,2.26]{Bibi15}. Any morphism in $\iMTM^\intt(X)$ between objects of $\IrrMHM(X)$ (that is, a morphism in $\IrrMHM(X)$) can be rescaled by using the same procedure as for $\cM',\cM''$. The rescaled morphism is then a morphism in $\MTM(\taucX,(*0_\tau))$, hence it is \sRspe along $\tau$, and graded (\cf\cite[Lem.\,2.27]{Bibi15}). It follows from \loccit\ that its kernel and cokernel also belong to $\IrrMHM(X)$. In other words, $\IrrMHM(X)$ is a full abelian subcategory of $\iMTM^\intt(X)$, and any morphism is strict with respect to the weight filtration. In particular, each graded object $\gr_\ell^W$ is a pure object of $\IrrMHM(X)$.

Furthermore, any morphism in $\iMTM^\intt(X)$ between objects of $\IrrMHM(X)$ induces a bi-strict morphism between the underlying $\ccD_X$-modules equipped with their irregular Hodge and weight filtration (\cf \cite[Prop.\,2.31 \& 2.32]{Bibi15}). As a consequence, for $\ccM\in\Mod(\cD_X)$ underlying $\cT\in\IrrMHM(X)$, the irregular Hodge filtration of $\gr_\ell^W\ccM$ is the filtration induced by $F_\bbullet^\irr\ccM$.

Since the construction of the Malgrange extension of \cite[Th.\,1.3]{Mochizuki21} is compatible with the standard functors, it follows that the properties proved for $\IrrMHM(X)$ in \cite[\S2.4]{Bibi15}, namely compatibility with projective pushforward and smooth pullback, hold with this improved definition of $\IrrMHM(X)$.

\subsection{The category $\protect\Cresc(X)$ and its properties}\label{subsec:compmochi}
We recall some results of \cite{Mochizuki21} on the category $\Cresc(X)$ and explain its relation with the category $\IrrMHM(X)$. We~\hbox{refer} to the monograph \cite{Mochizuki11} for the definition and the properties of (integrable) mixed twistor D-modules. We only consider the analytic setting in what follows.

Recall that, given an integrable mixed twistor D-module $((\cM',\cM'',C),W_\bbullet)$, the $\wt\cR_\cX$-module $\cM''$ is said to underlie this twistor D-module. In \cite{Mochizuki21}, T.\,Mochizuki introduces the category $\Cresc(X)$ as a subcategory of that of $\wt\cR_\cX$-modules underlying an integrable mixed twistor D-module. The objects of $\Cresc(X)$ are precisely those which satisfy the assumptions of Theorem \ref{th:wellresc}. The category $\Cresc(X)$ has the drawback of being non abelian. However, it is preserved by various functors on $\wt\cR_\cX$-modules (\cf \cite[Th.\,1.7]{Mochizuki21}. For example, if $H\subset X$ is a smooth hypersurface which is non-characteristic with respect to an object $\cM$ of $\Cresc(X)$, then the twistor-localization morphism $\cM\to\cM[*H]$ and the dual twistor-localization morphism $\cM\to\cM[!H]$ are morphisms in $\Cresc(X)$ whose kernel an cokernel belong to $\Cresc(X)$.

For an object $\cM$ of $\Cresc(X)$, the underlying $\ccD_X$-module $\ccM$ is equipped with a coherent filtration $F^\irr_\bbullet\ccM$ indexed by $A+\ZZ$ for some finite subset $A\subset[0,1)$, that we can regard as a nested family, indexed by $\alpha\in A$, of $\ZZ$-indexed coherent filtrations $F^\irr_{\alpha+\bbullet}\ccM$. We consider the various Rees modules $R_{F^\irr_\alpha}\ccM$ (these are holonomic graded $R_F\ccD_X$-modules). The good behavior of the objects of $\Cresc(X)$ with respect to some functors transfers to these Rees modules.

\subsubsection*{Duality and strict holonomicity}
In general, let $(\ccM,F_\bbullet\ccM)$ be a holonomic $\cD_X$-module equipped with a coherent $F$-filtration and let $R_F\ccM$ denote its Rees module, which is a graded module over the Rees graded ring $R_F\ccD_X$. We say that $R_F\ccM$ is \emph{strictly holonomic} if
\begin{enumeratea}
\item
$\cExt^i_{R_F\ccD_X}(R_F\ccM,R_F\ccD_X)=0$ for $i\neq n=\dim X$,
\item
and $\cExt^n_{R_F\ccD_X}(R_F\ccM,R_F\ccD_X)$ is strict.
\end{enumeratea}
The dual module $\bD(R_F\ccM)$, which is the graded $R_F\ccD_X$-module obtained by side-changing from $\cExt^n_{R_F\ccD_X}(R_F\ccM,R_F\ccD_X)$, takes then the form $R_F(\bD\ccM)$ for some filtration $F_\bbullet\bD(\ccM)$ on the dual holonomic $\ccD_X$-module $\bD(\ccM)$. Furthermore, biduality $\bD\bD(R_F\ccM)\simeq R_F\ccM$ holds and $\bD(R_F\ccM)$ is strictly holonomic.

\begin{proposition}[{\cite[Cor.\,6.55]{Mochizuki21}}]\label{prop:dual}
Let $\cM$ be an object of the category $\Cresc(X)$ and let $(\ccM,(F^\irr_{\alpha+\ZZ}\ccM)_{\alpha\in A})$ be the associated nested-filtered $\ccD_X$-module. Then
\begin{itemize}
\item
the complex $\bD\cM$ is concentrated in degree zero and its cohomology belongs to $\Cresc(X)$;
\item
the associated irregular Hodge filtration is indexed by $-A+\ZZ$;
\item
each $R_F\ccD_X$-module $R_{F^\irr_\alpha}\ccM$ is strictly holonomic, and the coherent dual filtration on the dual $\ccD_X$-module $\bD\ccM$ is the irregular filtration $F^\irr_{\beta+\ZZ}(\bD\ccM)$ induced by the object $\bD\cM$ of $\Cresc(X)$, where $\beta=({{<}-}\alpha-1)$ is the predecessor of $-\alpha-1$ in $-A+\ZZ$.
\end{itemize}
\end{proposition}

\subsubsection*{Non-characteristic inverse images}
If an $\cR_\cX$-module $\cM$ underlies a mixed twistor D\nobreakdash-module on $X$, its characteristic variety $\mathrm{Char}\cM\subset(T^*X)\times\CC_\hb$ is contained in the product $\Lambda\times\CC_\hb$, where $\Lambda$ is a Lagrangian closed analytic subset of $T^*X$, that one can take minimal for this property. For a morphism $f:X'\to X$, we say that~$f$ is non-characteristic with respect to $\cM$ if it is so with respect to $\Lambda$. Writing $f$ as the composition of its graph inclusion $\iota_f:X'\hto X'\times X$ and the projection $p:X'\times X\to X$, the non-charactericity condition is equivalent to that of $\iota_f$ with respect to $\Lambda\times T^*_XX$.

We say that $f$ is strictly non-characteristic with respect to $\cM$ if moreover the cohomology of its $\cR_\cX$-module pullback complex $\Dm f^*\cM$ is strict (\cf Notation \ref{subsec:nota}\eqref{nota:10}). This implies that $\Dm f^*\cM\simeq\Dm f^{*(0)}\cM$ is concentrated in degree zero.

As by definition $\cM$ is strict, \ie has no $\cO_{\CC_\hb}$-torsion, the same property holds for its $\cR_\cX$-module pullback $\Dm p^*\cM$. On the other hand, a non-characteristic closed inclusion is strictly non-characteristic with respect to $\cM$: as this is a local property, one iteratively reduce to the case of a codimension-one non-characteristic inclusion, in which case the strictness property follows from the strict specializability of $\cM$. In~conclusion, for $\cM$ underlying an object of $\MTM(X)$, the non-charactericity of $f$ is equivalent to its strict non-charactericity.

In the appendix, for the sake of completeness, we review the notion of non-characteristic pullback of a (possibly integrable) mixed twistor D-module as it is not explicitly defined in all cases in \cite{Mochizuki11}.

\begin{proposition}[{\cite[Prop.\,6.48 \& 6.67]{Mochizuki21}}]\label{prop:pull}
Let $\cM$ be an object of $\Cresc(X)$ with underlying nested-filtered $\cD_X$-module $(\ccM,(F^\irr_{\alpha+\ZZ}\ccM)_{\alpha\in A})$, and let $f:X'\to X$ be a holomorphic map which is non-characteristic with respect to $\cM$ (hence to $\ccM$). Then $\Dm f^*\cM$ is an object of $\Cresc(X')$ with underlying nested-filtered $\cD_{X'}$-module $(\Dm f^*\ccM,(F^\irr_{\alpha+\ZZ}(\Dm f^*\ccM))_{\alpha\in A})$, and we have for $\alpha\in A$:
\[
R_{F^\irr_\alpha}(\Dm f^*\ccM)\simeq \Dm f^*(R_{F^\irr_\alpha}\ccM).
\]
\end{proposition}
On the right-hand side, $\Dm f^*$ is taken in the sense of $R_F\cD_X$-modules. The isomorphism implies that the right-hand side is concentrated in degree zero and that it is strict, because the left-hand side is so. In other words, if $f$ is non-characteristic with respect to $\cM\in\Cresc(X)$, then it is strictly non-characteristic with respect to each $R_{F^\irr_\alpha}\ccM$ ($\alpha\in A$).

\begin{proof}
The result of \cite[Prop.\,6.67]{Mochizuki21} asserts that
\[
F^\irr_{\alpha+\ZZ}(\Dm f^*\ccM)=f^*(F^\irr_{\alpha+\ZZ}\ccM)
\]
for each $\alpha\in A$. It remains to be proved that $\Dm f^*(R_{F^\irr_\alpha}\ccM)$ is concentrated in degree zero and is strict. We write $f$ as the composition of a non-characteristic closed inclusion and a projection. The case of a projection is easy (\cf\eg\cite[Rem.\,8.6.7]{S-Sch}), and since the question is local, the case of a non-characteristic closed inclusion redu\-ces by induction on the codimension to the case of codimension one. Then the result is furnished by \cite[Prop.\,6.48]{Mochizuki21}, which asserts that each $R_{F^\irr_\alpha}\ccM$ is strictly specializable along a non-characteristic smooth hypersurface.
\end{proof}

\Subsubsection*{Projective pushforward}

\begin{proposition}[{\cite[Prop.\,6.24 \& 6.46]{Mochizuki21} and \cite[Th.\,2.62]{Bibi15}}]\label{prop:push}
Let $f:X\to X'$ be a projective morphism of complex manifolds and let $\cM$ be an object of $\Cresc(X)$. Then the pushforward modules $\Dm f_*^{(k)}\cM$ are objects of $\Cresc(X')$ and we have for each $a\in\ZZ$:
\[
\Dm f_*^{(k)}(R_{F^\irr_\alpha}\ccM)\simeq R_{F^\irr_\alpha}(\Dm f_*^{(k)}(\ccM)),
\]
in particular, the spectral sequence attached to the pushforward of the filtered $\cD_X$\nobreakdash-mod\-ule $(\ccM,F^\irr_{\alpha+\ZZ}\ccM)$ degenerates at $E_1$.
\end{proposition}

\subsubsection*{Application to $\IrrMHM(X)$}
It is straightforward to check that the $\wt\cR_\cX$-module underlying an object of $\IrrMHM(X)$ as defined in Section \ref{subsec:improved} is an object of $\Cresc(X)$. In~particular, the notion of irregular Hodge filtration of \cite{Bibi15} coincides with that of \cite[Cor.\,1.6 \& \S6.5]{Mochizuki21}. It follows from the results of \cite[Th.\,1.7]{Mochizuki21} that the irreg\-ular Hodge filtration induced by an object of $\IrrMHM(X)$ satisfies the compatibility properties of \loccit

\section{Vanishing theorems}

\subsection{A criterion for the Kodaira-Saito vanishing property}
We recall the result proved in \cite[\S11.9]{S-Sch}, following the proof of \cite{MSaito87}. For the ample line bundle $L$ on $X$ we choose an integer $m\geq2$ such that $L^{\otimes m}$ defines an embedding $X\hto\PP^N$ and we let $\iota_H:H\hto X$ denote a hyperplane section. We say that $(\ccM,F_\bbullet\ccM)$ satisfies \emph{the Kodaira-Saito vanishing property} (with respect to $L$) if
\begin{align*}
H^i(X,\gr^F\pDR(\ccM)\otimes L)&=0\quad \text{for }i>0,\\
H^i(X,\gr^F\pDR(\ccM)\otimes L^{-1})&=0\quad \text{for }i<0.
\end{align*}

Recall that we denote by $\Dm f_*^{(k)}$, \resp $\Dm f^{*(k)}$, the $k$-th $\ccD$- or $R_F\ccD$-module pushforward, \resp pullback, by the map $f$, and by $a_X$ the constant map $X\to\mathrm{pt}$.

Classical constructions of coverings (\cf \cite[\S4.b]{Lazarsfeld04} and \cite[\S2]{E-V86}) produce a finite morphism $f:X'\to X$ satisfying the following properties:
\begin{enumeratea}
\item\label{subsec:cycliccovering1}
the source $X'$ is smooth, as well as $H':=f^{-1}(H)$,
\item\label{subsec:cycliccovering2}
the restriction $f:H'\to H$ is an isomorphism,
\item\label{subsec:cycliccovering3}
setting $U=X\moins H$ and $U'=f^{-1}(U)=X'\moins H'$, the restriction $f:U'\to U$ is a degree $m$ covering, and $f$ is cyclically ramified along $H$,
\item\label{subsec:cycliccovering4}
the bundle $f^*L$ is very ample, $H':=f^{-1}(H)\simeq H$ is a corresponding hyperplane section of $X'$, and $U'$ is affine,
\item\label{subsec:cycliccovering5}
there exists a canonical isomorphism $f_*\ccO_{X'}\simeq\bigoplus_{i=0}^{m-1}L^{-i}$ (with $L^0:=\ccO_X$).
\end{enumeratea}

The next criterion is proved in \cite[\S11.9]{S-Sch} by following the proof of M.\,Saito in \cite{MSaito87} for mixed Hodge modules.

\begin{theoreme}\label{th:KSproperty}
Let $(\ccM,F_\bbullet\ccM)$ be a coherently filtered $\ccD_X$-module which is strictly holonomic. Assume that there exists a hyperplane section $H$ (relative to $L^{\otimes m}$) which is strictly non-characteristic for $R_F\ccN=R_F\ccM$ or $R_F\bD(\ccM)$ such that
\begin{enumerate}
\item\label{th:KSproperty1prime}
with respect to the associated cyclic covering $f:X'\to X$, the pushforward $\CC[\hb]$-modules $\Dm a_{X'*}^{(k)}(\Dm f^{*(0)}(R_F\ccN))$ are torsion-free;
\item\label{th:KSproperty2}
the restriction $\Dmiota_{H*}(R_F\ccN_H):=\Dmiota_{H*}(\Dmiota^*_HR_F\ccN)$ satisfies the Kodaira-Saito vanishing property.
\end{enumerate}

Then $R_F\ccM$ and $R_F\bD(\ccM)$ satisfy the Kodaira-Saito vanishing property.\qed
\end{theoreme}

\Subsection{Proof of Theorem \ref{th:vanishing} and Corollaries \ref{cor:vanishing3} and \ref{cor:vanishing4}}

\begin{proof}[Proof of Theorem \ref{th:vanishing}]
We consider the $\ZZ$-indexed filtration $F^\irr_{\alpha+\ZZ}\ccM$ for $\alpha\in A$ and we~prove the theorem by induction on the dimension of the support of~$\ccM$ (equivalently, that of $\cM$), the case when it has dimension zero being clear. As~remarked above, $\cM$ is an object of $\Cresc(X)$, so we can apply the results of Section~\ref{subsec:compmochi} to the filtered $\cD_X$-module $(\ccM,F^\irr_{\alpha+\ZZ}\ccM)$ ($\alpha\in A$). We first argue with $\cM$.

We first note that the cyclic covering $f:X'\to X$ is (strictly) non-characteristic with respect to $\cM$. By Proposition \ref{prop:pull}, $\Dm f^*(R_{F^\irr_\alpha}\ccM)$ is the Rees module of the irregular Hodge filtration indexed by $\alpha+\ZZ$ attached to the object $\Dm f^*\cM$ of $\Cresc(X')$. Property \ref{th:KSproperty}\eqref{th:KSproperty1prime} then follows from Proposition \ref{prop:push} applied to the constant map $a_{X'}$ and to $\Dm f^*\cM$. Property \eqref{th:KSproperty2} is proved by induction on the dimension of the support of~$\cM$. It is then a direct consequence of \cite[Prop.\,6.48]{Mochizuki21}, as reviewed in Proposition \ref{prop:pull}.

For $\bD\cM$, according to Proposition \ref{prop:dual}, the above argument for $\cM$ and $R_{F^\irr_\alpha}\ccM$ also applies to $\bD\cM$ and $R_{F^\irr_\beta}\bD\ccM$ for $\beta=({{<}-}\alpha-1)$. Therefore, the conclusion~of Theorem \ref{th:KSproperty} holds for each filtered $\cD_X$-module $(\ccM,F^\irr_{\alpha+\ZZ}\ccM)$ ($\alpha\in A$) and the dual filtered $\cD_X$-module $(\bD\ccM,F^\irr_{\beta+\ZZ}\bD\ccM)$ ($\beta\in B=(-A+\ZZ)\cap[0,1)$). This concludes the proof of Theorem \ref{th:vanishing}.
\end{proof}

\begin{proof}[Proof of Corollary \ref{cor:vanishing3}]
We recall that, for $a\in A+\ZZ$,
\[
F^\irr_{a-n}\pDR(\ccM)=\Bigl\{F^\irr_{a-n}(\ccM)\ra\Omega^1_X\otimes F^\irr_{a-n+1}(\ccM)\ra\cdots\ra\Omega^n_X\otimes F^\irr_a(\ccM)\Bigr\},
\]
where $\Omega^n_X\otimes F^\irr_a(\ccM)=\omega_X\otimes F^\irr_a(\ccM)$ is in degree zero. For $a\in[a_o,a_o+1)$, this complex reduces to $\omega_X\otimes F^\irr_a(\ccM)$ in degree zero, and the desired vanishing is nothing but the first line in Theorem \ref{th:vanishing}.
\end{proof}

\begin{proof}[Proof of Corollary \ref{cor:vanishing4}]
We start with a general lemma.

\begin{lemme}\label{lem:general}
Let $f:X\to Y$ be a proper morphism between two smooth complex manifolds of respective dimensions $n$ and $m$, and let $(\ccM,F_\bbullet\ccM)$ be a coherently filtered (left) $\ccD_X$-module with associated Rees module $R_F\ccM$. Assume that
\begin{enumerate}
\item\label{lem:general1}
$p$ is an index such that $F_{p-1}\ccM=0$;
\item\label{lem:general2}
each $\Dm f_*^{(j)}(R_F\ccM)$ is strict, \ie there exists a (unique) coherent filtration $F_\bbullet(\Dm f_*^{(j)}\ccM)$ such that $\Dm f_*^{(j)}(R_F\ccM)=R_F(\Dm f_*^{(j)}\ccM)$.
\end{enumerate}
Then, for each $j\in\ZZ$, we have
\[
\omega_Y\otimes F_{p+m-n}(\Dm f_*^{(j)}\ccM)
\simeq R^jf_*(\omega_X\otimes F_p\ccM)\quand F_{p+m-n-1}(\Dm f_*^{(j)}\ccM)=0.
\]
\end{lemme}

\begin{proof}
We use the following formula for the pushforward of a coherent $\cD_X$-module (\cf \eg \cite[Ex.\,8.51]{S-Sch}):
\[
\omega_Y\otimes\Dm f^{(j)}_*\ccM\simeq\ccH^j\Bigl(\bR f_*\pDR(\ccM\otimes_{f^{-1}\cO_Y}f^{-1}\ccD_Y)\Bigr)
\]
and its analogue for the associated Rees modules. According to \eqref{lem:general2}, for each $p\in\ZZ$, the term of $\hb$-degree $p-n$ in $\omega_Y\otimes \Dm f^{(j)}_*(R_F\ccM)$ is $\omega_Y\otimes F_{p+m-n}(\Dm f^{(j)}_*\ccM)$. On the other hand, the term in degree $p-n$ in $\ccH^j\bigl(\bR f_*\pDR R_F(\ccM\otimes_{f^{-1}\cO_Y}f^{-1}\ccD_Y)\bigr)$ is $\ccH^j\bigl(\bR f_*F_{p-n}\pDR(\ccM\otimes_{f^{-1}\cO_Y}f^{-1}\ccD_Y)\bigr)$. We then find for each $p\in\ZZ$:
\begin{equation}\label{eq:FfF}
\omega_Y\otimes F_{p+m-n}(\Dm f^{(j)}_*\ccM)\simeq\ccH^j\Bigl(\bR f_*F_{p-n}\pDR(\ccM\otimes_{f^{-1}\cO_Y}f^{-1}\ccD_Y)\Bigr).
\end{equation}
Let us choose $p$ as in \eqref{lem:general1}. Then
\[
F_{p-n}\pDR(\ccM\otimes_{f^{-1}\cO_Y}f^{-1}\ccD_Y)=\omega_X\otimes F_p\ccM,
\]
and the right-hand side of \eqref{eq:FfF} reads $R^jf_*(\omega_X\otimes F_p\ccM)$. Furthermore,
\[
F_{p-n-1}\pDR(\ccM\otimes_{f^{-1}\cO_Y}f^{-1}\ccD_Y)=0,
\]
and this concludes the proof.
\end{proof}

\subsubsection*{End of the proof of Corollary \ref{cor:vanishing4}}
Property \ref{lem:general}\eqref{lem:general2} holds because of Proposition \ref{prop:push}. We can thus apply the lemma to $F^\irr_{\alpha+\bbullet}\ccM$ for each $\alpha\in A$. Then, Corollary \ref{cor:vanishing3} applied to $\Dm f_*^{(j)}\cT$ and $\Dm f_*^{(j)}\ccM$ with its irregular Hodge filtration $F^\irr_\bbullet(\Dm f_*^{(j)}\ccM)$ yields the result.
\end{proof}

\Subsection{Proof of Corollaries \ref{cor:vanishing} and \ref{cor:vanishing5}}
\begin{proof}[Proof of Corollary \ref{cor:vanishing}]
We keep the assumptions and notation of Corollary~\ref{cor:vanishing} that we now prove. To the morphism $\varphi:X\to\PP^1$ one can associate (\cf\cite[Th.\,0.2]{Bibi15}) an object $\cT^{\varphi/\hb}[*D]$ of $\IrrMHM(X)$, whose underlying $\ccD_X$\nobreakdash-mod\-ule is $(\ccO_X(*D),\rd+\rd \varphi)$: it is defined as the tensor product of the object $\cT^{\varphi/\hb}$ and the mixed Hodge module $\ccO_X^\rH[*H]$, if $H$ consists of the components of $D$ not in $P$. The irregular Hodge filtration of $\ccO_X(*D)$ that is associated to it by \cite[Th.\,0.3]{Bibi15} induces on its de~Rham complex the irregular Hodge filtration considered in \cite{E-S-Y13}, according to \cite[Th.\,1.3(5)]{S-Y14}.

For any fixed $\alpha\in A$, the filtered complex $(\Omega^\cbbullet(\log D,\varphi,\alpha),\rd+\rd \varphi,\sigma)$ (where $\sigma$ denotes the filtration by stupid truncation) is filtered quasi-isomorphic to the complex $F_\alpha^{\irr,\cbbullet}\DR(\ccO_X(*D),\rd+\rd \varphi))$, according to \cite[Cor.\,1.4.5]{E-S-Y13}. We thus have
\[
\frac{F_\alpha^{\irr,p}}{F_\alpha^{\irr,p+1}}\bigl[\DR(\ccO_X(*D),\rd+\rd \varphi))\bigr]\simeq \Omega^p(\log D,\varphi,\alpha)[-p]
\]
for any $\alpha\in A$. The assertion follows then from Theorem \ref{th:vanishing}.
\end{proof}

\begin{proof}[Proof of Corollary \ref{cor:vanishing5}]
We apply Corollary \ref{cor:vanishing4} to $(\cO_X(*D),\rd+\rd\varphi)$. It follows from \cite[(1.6.2)]{E-S-Y13} that the minimal values of $F_\bbullet^\irr(\cO_X(*D),\rd+\rd\varphi)$ are equal to $\cO_X(D+\nobreak\lfloor\alpha P\rfloor)$ for $\alpha\in A$.
\end{proof}

\appendix

\section{Non-characteristic pullback in\\ mixed~twistor~\texorpdfstring{D\protect\nobreakdash-module}{Dmod} theory}

We recall that a mixed twistor $D$-module on $X$, that is, an object of $\MTM(X)$, is a $W$-filtered triple $((\cM',\cM'',C),W_\bbullet)$ satisfying various properties. Here, $\cM',\cM''$ are holonomic $\cR_\cX$-modules and, denoting by $\bS\subset\CC_\hb$ the unit circle, $C$ is a sesquilinear pairing $\cM'\otimes\sigma^*\ov{\cM''}\to\Db_{X\times\bS/\bS}$ with $\sigma$ being the involution $\hb\mto-1/\ov\hb$ and $\Db_{X\times\bS/\bS}$ the sheaf of distributions on $X\times\bS$ which depend continuously of $\hb\in\bS$. We aim at making precise the definition the strictly non-characteristic pullback of such an object, being understood that the case of the restriction to a strictly non-characteristic smooth principal divisor has been defined in \cite{Bibi01c} by means of the nearby cycle functor.

We deal with holonomic $\cR_\cX$-modules, that is, coherent $\cR_\cX$-modules whose characteristic variety is contained in a product $\Lambda\times T^*\CC_\hb$, with $\Lambda$ being a Lagrangian conic closed analytic subset of $T^*X$. The pullback of a holonomic $\cR_\cX$-module $\cM$ by a morphism $f:X'\to X$ is well-defined, and coherence is preserved if $f$ is non-characteristic with respect to $\cM$ (\cf\eg\cite[Def.\,4.6]{Kashiwara03}). Assume that $\cM$ is strict, \ie has no $\cO_{\CC_\hb}$-torsion. We say that $f$ is \emph{strictly non-characteristic with respect to~$\cM$} if it is non-characteristic and if the $\cR_X$-module pullback $\Dm f^*\cM$ has strict cohomology. Since the non-characteristic pullback $\Dm f^*\ccM$ of a holonomic $\ccD_X$-module is concentrated in degree zero, a strictly non-characteristic pullback of a strict holonomic $\cR_\cX$-module is also concentrated in degree zero.

In order to define the strictly non-characteristic pullback of a triple $(\cM',\cM'',C)$, we just need to define the pullback of the sesquilinear pairing $C$ and to show that, in the case of a smooth principal divisor, it coincides with the pullback obtained by means of nearby cycles. By writing a morphism as the composition of its graph inclusion followed by the projection, we only need to consider these two cases separately.

Let $W$ be an open subset of $X\times\bS$. A relative distribution $T\in\Db_{X\times\bS/\bS}(W)$ is by definition a $C^\infty(\bS)$-linear map $\cE^{(n,n)}_{X\times\bS/\bS,\rc}(W)\to C^0(\bS)$ which satisfies a usual continuity property with respect to the sup norm on $C^0(\bS)$.

If $f:X'\to X$ is a smooth map and $T\in\Db_{X\times\bS/\bS}(W)$, then, setting $W'=(f\times\id)^{-1}(W)$, the pullback $f^*T\in\Db_{X\times\bS/\bS}(W')$ is defined in such a way that, for $\eta\in \cE^{(n',n')}_{X\times\bS/\bS,\rc}(W')$, $f^*T(\eta)=T(\int_f\eta)$, on noting that, because $f$ is smooth, $\int_f\eta\in\cE^{(n,n)}_{X\times\bS/\bS,\rc}(W)$.

We are left with the case where $f:X'\hto X$ is a closed immersion. We recall that a relative distribution $T\in\Db_{X\times\bS/\bS}(W)$ is the limit of a sequence $T_n\in C^{\infty,0}_\rc(W)$ ($C^\infty$~with respect to the $X$-variable and continuous with respect to the $\bS$-variable), that is, for each $\eta\in\cE^{(n,n)}_{X\times\bS/\bS,\rc}(W)$, we have
\[
\lim_n\int_X T_n\cdot\eta = T(\eta)\in C^0(\bS).
\]

Set $W'=W\cap X'$. We say that \emph{$T$ can be restricted to $W'$} if there exists a relative distribution $T'\in\Db_{X'\times\bS/\bS}(W')$ such that, for any sequence $T_n\in C^{\infty,0}_\rc(W)$ converging to $T$ in $\Db_{X\times\bS/\bS}(W)$, the sequence $T_n|_{W'}\in C^{\infty,0}_\rc(W')$ converges to $T'$ in $\Db_{X'\times\bS/\bS}(W')$. If it exists, such a $T'$ is unique.

\begin{proposition}\label{prop:restrictC}
For any sections $m'\in\cM'(W)$ and $m''\in\cM''(W)$, the distribution $T=C(m',m'')\in\Db_{X\times\bS/\bS}(W)$ can be restricted to $W'$.
\end{proposition}

\begin{lemme}
Assume that $X=X'\times\CC_t$ and let $(\cM',\cM'',C)$ be a holonomic triple which is strictly non-characteristic along $X'=\{t=0\}$. For local sections $m',m''$ of $\cM',\cM''$ on an open subset $W\subset X\times\bS$, the relative distribution $T=C(m',m'')\in\Db_{X\times\bS/\bS}(W)$ can be restricted to $W'$ and its restriction is equal to $T'=\psi_{t,-1}T$ as defined by the residue formula \cite[(3.6.10)]{Bibi01c}.
\end{lemme}

\begin{proof}[Sketch of proof]
Let $T_n\in C^{\infty,0}_\rc(W)$ be a sequence converging to $T$ in $\Db_{X\times\bS/\bS}(W)$. One shows that the residue formula \cite[(3.6.10)]{Bibi01c} applied to $T_n$ yields $T_n|_{W'}$ and that the residue formula passes to the limit $n\to\infty$.
\end{proof}

\begin{proof}[Proof of Proposition \ref{prop:restrictC}]
The question is local as the restriction is unique if it exists. We can thus assume that $W=X'\times\Delta^r$ with coordinates $t_1,\dots,t_r$ on $\Delta^r$. We argue by induction on $r$. By the lemma, the proposition holds if $r=1$. Assume thus that $r\geq2$. Set $\iota:X'_1=\{t_1=0\}\hto X$ and $f_1:X'\hto X'_1$. Then ${X'_1}$ is strictly non-characteristic for $\cM',\cM''$, and $X'$ is strictly non-characteristic for $\Dm\iota^*_{X'_1}\cM',\Dm\iota^*_{X'_1}\cM''$. By the case $r=1$, $T$ can be restricted to ${X'_1}$ in some neighborhood of $X'$, and by induction $\iota^*_{X'_1}T$ can be restricted to $X'$. It follows that $T$ can be restricted to $X'$ and $f^*T=f_1^*\bigl(\iota^*T\bigr)$.
\end{proof}

\section{Non-integral gradings}\label{app:B}
In this section, we indicate the arguments for obtaining the assertion in Remark \ref{rem:vanishing}\eqref{rem:vanishing2}.

We extend the Rees construction to filtrations indexed by $A+\ZZ$, where $A$ is a finite subset of $[0,1)$ with $\#A=r$. We assume that $A$ contains $0$, so that we have a (unique) increasing bijection $A\simeq\{0,1/r,\dots,(r-1)/r\}$. In the following, we identify these two sets.

We consider the ring $\CC[\hu]$ with the subring $\CC[\hb]\hto\CC[\hu]$ so that $\hb$ is mapped to~$\hu^r$. The variable~$\hu$ is given the degree $1/r$. Recall that $R_X$ denotes the graded ring $R_F\ccD_X$. We~set $R_X^{(r)}=\CC[\hu]\otimes_{\CC[\hb]}R_X$. This is a $\frac1r\ZZ$-graded ring containing~$R_X$ as a $\ZZ$-graded subring, with term of degree $p/r$ given by
\[
(R_X^{(r)})_{p/r}=(R_X)_{\lfloor p/r\rfloor}.
\]
In other words, $R_X^{(r)}$ is the Rees ring of $\ccD_X$ with respect to the filtration
\[
F_p^{(r)}\ccD_X=F_{\lfloor p/r\rfloor}\ccD_X.
\]

\begin{proposition}\label{prop:Mir}
Giving a $\frac1r\ZZ$-graded $R_X^{(r)}$-module $M^{(r)}$ is equivalent to giving a~\hbox{finite} family $M_{i/r}$ ($i=0,\dots,r-1$) in $\Modgr(R_X)$ together with morphisms $M_{(i-1)/r}\to M_{i/r}$ ($i=1,\dots,r-1$) and $M_{(r-1)/r}\to M_1(1)$ such that, for each $i=0,\dots,r$, their composition $M_{i/r}\to M_{1+i/r}$ is equal to the multiplication by $\hb$. As an $R_X$-module, $M^{(r)}$ decomposes as $\bigoplus_{i=0}^{r-1}M_{i/r}\otimes u^i$.

Furthermore, $M^{(r)}$ is strict (\ie $\CC[\hu]$-flat) if and only if each $M_{i/r}$ is strict (\ie $\CC[\hb]$-flat). Lastly, $M^{(r)}$ is $R_X^{(r)}$-coherent if and only if each $M_{i/r}$ is $R_X$-coherent.
\end{proposition}

\begin{proof}
For $i=0,\dots,r-1$, we consider the $\ZZ$-graded objects $M^{(r)}_{i+r\ZZ}$. These are $\ZZ$\nobreakdash-graded $R_X$-modules, that we denote by $M_{i/r}$. The morphism $\hu:M^{(r)}_j\to M^{(r)}_{j+1}$ induces the desired family of morphisms.

Conversely, from the family $M_{i/r}$ and the morphisms, we set, for $p=qr+i$ with $i\in\{0,\dots,r-1\}$, $M^{(r)}_{p/r}:=(M_{i/r})_q$ and the morphisms $M_{(i-1)/r}\to M_{i/r}$ ($i=1,\dots,r-1$) and $M_{(r-1)/r}\to M_1(1)$ are interpreted as the multiplication by $\hu$.

The flatness statement is then clear since the set of elements of $\CC[\hu]$-torsion is equal to that of $\CC[\hb]$-torsion in $M^{(r)}$, and the last statement follows \eg from \cite[Prop.\,A.10]{Kashiwara03}.
\end{proof}

If $M^{(r)}$ is $\CC[u]$-flat, it is equal to the Rees module of some $\ccD_X$-module $\ccM$ with respect to a $\frac1r\ZZ$-indexed $F^{(r)}$-filtration $F^{(r)}_\bbullet\ccM$, that we can consider as a family of nested $\ZZ$-indexed $F$-filtrations $F_{i/r+\bbullet}\ccM$ with Rees module $M_{i/r}$, \ie satisfying
\[
i/r+p\leq j/r+q\implies F_{i/r+p}\ccM\subset F_{j/r+q}\ccM\quad\forall p,q\in\ZZ,\;\forall i,j\in\{0,\dots,r-1\}.
\]
We then have
\[
\gr^{F^{(r)}}\ccM=\bigoplus_{a\in\frac1r\ZZ}F^{(r)}_a\ccM/F^{(r)}_{<a}\ccM=\bigoplus_{a\in\frac1r\ZZ}\gr^{F^{(r)}}_a(\ccM).
\]

We say that $F^{(r)}_\bbullet(\ccM)$ is a coherent $F^{(r)}$-filtration if $M^{(r)}$ is $R_X^{(r)}$-coherent, and this property is equivalent to each $F_{i/r+\bbullet}(\ccM)$ being a coherent $F$-filtration of $\ccM$.

Conversely, given a family of nested $\ZZ$-indexed $F$-filtrations $F_{i/r+\bbullet}(\ccM)$, we claim that the Rees module $R_{F^{(r)}}(\ccM)=\bigoplus_{k\in\ZZ}F^{(r)}_{k/r}(\ccM)\cdot\hu^k$ is an $R_X^{(r)}$-module. Indeed, we~need to prove that $F_{\lfloor k/r\rfloor}(\ccD_X)\cdot F^{(r)}_{\ell/r}(\ccM)\subset F^{(r)}_{(k+\ell)/r}(\ccM)$: this follows from the inequality $k+\ell\geq r\lfloor k/r\rfloor+\ell$.

\begin{lemme}\label{lem:ramifF}
For any $(\ccM,F^\irr_{\alpha+\ZZ}\ccM)$ as in Theorem \ref{th:vanishing}, there exists an isomorphism in $\catD^\rb_\coh(\ccO_X)$:
\[
\bD(\gr^{F^\irr}\pDR(\ccM))\simeq\gr^{F^\irr}\pDR(\bD\ccM).
\]
\end{lemme}

\begin{proof}
We can adapt the proof for $R_F\ccD_X$-modules as given \eg in \cite[Prop.\,8.8.31]{S-Sch}, by replacing the ring $R_F\ccD_X$ (denoted there by $\wt\ccD_X$) with $R_X^{(r)}$. We~only need to check that $M^{(r)}$ is strictly holonomic as an $R_X^{(r)}$-module, knowing that each $M_{i/r}$ is strictly holonomic as an $R_X$-module (because of Proposition \ref{prop:dual}).

We can regard $M^{(r)}$ as an $R_X$-module, and it decomposes as such as the direct sum of the $R_X$-modules $M_{i/r}$ ($i=0,\dots,r-1$), and so $\RcHom_{R_X}(M^{(r)},R_X)$ has cohomology in degree $n$ at most. The same property holds then for
\begin{align*}
\RcHom_{R_X}(M^{(r)},R_X)\otimes^\bL_{\CC[\hb]}\CC[\hu]&\simeq\RcHom_{R_X}(M^{(r)},R_X^{(r)})\\
&\simeq\RcHom_{R_X^{(r)}}(\CC[\hu]\otimes_{\CC[\hb]}M^{(r)},R_X^{(r)}).
\end{align*}
It is thus enough to check that $M^{(r)}$ is a direct summand of $\CC[\hu]\otimes_{\CC[\hb]}M^{(r)}$, \ie to find a section $M^{(r)}\to\CC[\hu]\otimes_{\CC[\hb]}M^{(r)}$ of the natural surjective morphism $\CC[\hu]\otimes_{\CC[\hb]}M^{(r)}\to M^{(r)}$.

We write
\[
\CC[\hu]\otimes_{\CC[\hb]}M^{(r)}=\bigoplus_{k\in\ZZ}\Bigl(\bigoplus_\alpha F^\irr_{\alpha+\lfloor k/r-\alpha\rfloor}(\ccM)\Bigr)\cdot u^k,
\]
where $\alpha$ runs in $\{0,1/r,\dots,(r-1)/r\}$, and the surjective morphism to $M^{(r)}$ is indu\-ced~by
\[
\bigoplus_\alpha F^\irr_{\alpha+\lfloor k/r-\alpha\rfloor}(\ccM)\to\sum_\alpha F^\irr_{\alpha+\lfloor k/r-\alpha\rfloor}(\ccM)\subset\ccM.
\]
We have $\alpha+\lfloor k/r-\alpha\rfloor\leq k/r$ with equality if and only if $\alpha=\{k/r\}$ (fractional part), so that the right-hand side above is equal to $F^\irr_{k/r}(\ccM)$. Since $F^\irr_{k/r}(\ccM)$ is a summand in the left-hand side, we obtain the desired section.
\end{proof}

That the statement of Theorem \ref{th:KSproperty} holds when we replace coherent $F_\bbullet\ccD_X$-filtrations with coherent $F^{(r)}_\bbullet\ccD_X$-filtrations is obtained by applying the same method as in \cite[\S11.9.d]{S-Sch}, since Lemma \ref{lem:ramifF} reduces the proof to the vanishing of $H^i(X,\gr^{F^{(r)}}\pDR(\ccN)\otimes L)$ for $i>0$ and $\ccN=\ccM$ or $\bD\ccM$.\qed

\begin{proof}[Proof of Remark \ref{rem:vanishing}\eqref{rem:vanishing3}]
We first note that Assumption \eqref{lem:general2} of Lemma ~\ref{lem:general} also implies that \eqref{eq:FfF} also holds with $\gr^F$ instead of $F$. The irregular Hodge filtration indexed by $A+\ZZ$ satisfies Assumption \eqref{lem:general2} of Lemma ~\ref{lem:general} when the latter is extended to such kinds of filtrations by means of the equivalence of Proposition \ref{prop:Mir}. The previous remark applies then to the graded objects $\gr^{F^\irr}_a$ for $a\in A+\ZZ$.

Then for $a\in[a_o,a_o+1)$, we find
\[
\omega_Y\otimes \gr^{F^\irr}_{a+m-n}(\Dm f^{(j)}_*\ccM)\simeq R^jf_*(\omega_X\otimes \gr^{F^\irr}_a\ccM).
\]
We conclude by applying Corollary \ref{cor:vanishing3}, as we already noticed that one can replace~$F^\irr_a$ with $\gr^{F^\irr}_a$ in this corollary.
\end{proof}

\backmatter
\providecommand{\SortNoop}[1]{}\providecommand{\sortnoop}[1]{}\providecommand{\eprint}[1]{\href{http://arxiv.org/abs/#1}{\texttt{arXiv\string:\allowbreak#1}}}\providecommand{\hal}[1]{\href{https://hal.archives-ouvertes.fr/hal-#1}{\texttt{hal-#1}}}\providecommand{\tel}[1]{\href{https://hal.archives-ouvertes.fr/tel-#1}{\texttt{tel-#1}}}\providecommand{\doi}[1]{\href{http://dx.doi.org/#1}{\texttt{doi\string:\allowbreak#1}}}\providecommand{\didotfam}{}
\providecommand{\bysame}{\leavevmode ---\ }
\providecommand{\og}{``}
\providecommand{\fg}{''}
\providecommand{\cdrandname}{\&}
\providecommand{\cdredsname}{\'eds.}
\providecommand{\cdredname}{\'ed.}
\providecommand{\cdrmastersthesisname}{M\'emoire}
\providecommand{\cdrphdthesisname}{Th\`ese}
\providecommand{\eprint}[1]{\href{http://arxiv.org/abs/#1}{\texttt{arXiv\string:\allowbreak#1}}}
\providecommand{\eprintother}[3]{\href{#1/#2}{\texttt{#3\string:\allowbreak#2}}}

\end{document}